\documentclass[twoside]{article}

\usepackage[accepted]{aistats2025}
\usepackage[round]{natbib}

\newcommand{\R}{{\mathbb R}}
\newcommand{\N}{{\mathbb N}}
\newcommand{\Var}{\text{Var}}
\newcommand{\E}{\mathbb E}
\renewcommand{\P}{\mathcal P}
\newcommand{\sgn}{\text{sgn}}

\usepackage{times}
\usepackage{dsfont}
\usepackage{graphicx}
\usepackage{aliascnt}
\usepackage{amssymb}
\usepackage[colorlinks=true,linkcolor=blue,citecolor=blue]{hyperref}

\usepackage{amsthm}
\newtheorem{theorem}{Theorem}
\newtheorem{lemma}{Lemma}

\begin{document}

\twocolumn[

\aistatstitle{Efficient Estimation of a Gaussian Mean  with Local Differential Privacy}

\runningtitle{Efficient Estimation of a Gaussian Mean  with Local Differential Privacy}

\aistatsauthor{ Nikita Kalinin \And Lukas Steinberger }

\aistatsaddress{Institute of Science and Technology Austria (ISTA) \And  University of Vienna } ]

\begin{abstract}
  In this paper, we study the problem of estimating the unknown mean $\theta$ of a unit variance Gaussian distribution in a locally differentially private (LDP) way. In the high-privacy regime ($\epsilon\le 1$), we identify an optimal privacy mechanism that minimizes the variance of the estimator asymptotically. Our main technical contribution is the maximization of the Fisher-Information of the sanitized data with respect to the local privacy mechanism $Q$. We find that the exact solution $Q_{\theta,\epsilon}$ of this maximization is the sign mechanism that applies randomized response to the sign of $X_i-\theta$, where $X_1,\dots, X_n$ are the confidential iid original samples. However, since this optimal local mechanism depends on the unknown mean $\theta$, we employ a two-stage LDP parameter estimation procedure which requires splitting agents into two groups. The first $n_1$ observations are used to consistently but not necessarily efficiently estimate the parameter $\theta$ by $\tilde{\theta}_{n_1}$. Then this estimate is updated by applying the sign mechanism with $\tilde{\theta}_{n_1}$ instead of $\theta$ to the remaining $n-n_1$ observations, to obtain an LDP and efficient estimator of the unknown mean.%
\end{abstract}

\section{Introduction}

We consider the problem of estimating the unknown mean of a unit variance normal distribution in a locally differentially private and statistically efficient way. In this scenario, we have $n$ agents or data-holders, each of which owns data $X_i$ sampled independently from the probability distribution $P_\theta = N(\theta, 1)$ which depends on the unknown mean parameter $\theta \in \mathbb{R}$. Denote the class of all those potential data generating distributions by $\P=\{P_\theta:\theta\in\R\}$. We protect the privacy of the data owners by local differential privacy. This means that each agent generates a sanitized version $Z_i$ of their original data $X_i$ independently of everybody else by applying a privacy mechanism $Q$, which is a Markov kernel or a conditional distribution of $Z_i$ given $X_i=x$. In other words, $Q(A|x) = P(Z_i\in A|X_i=x)$. From the privacy mechanism $Q$ we can generate random variates in some arbitrary space $\mathcal{Z}$ endowed with a sigma algebra $\mathcal{G}$. We say that $Q$ satisfies the $\epsilon$-local differential privacy property if
\begin{equation}
Q(A|x) \le e^\epsilon Q(A|x')
\end{equation}
for all $x,x'\in\R$ and all events $A\in\mathcal{G}$. We denote by $\mathcal{Q}_\epsilon = \mathcal{Q}_\epsilon(\R)$ the set of all possible such privacy mechanisms taking inputs from $\R$.

After each agent applied the privacy mechanism $Q$ we obtain iid samples $Z_1, \dots Z_n$ from the distribution $QP_\theta$, which formally is the probability measure obtained by integrating $Q$ with respect to $P_\theta(dx)$. We can only use these samples to estimate the unknown parameter $\theta$ and our goal is to do this with the smallest amount of estimation variance possible, that is, to estimate $\theta$ in a statistically efficient way. Standard asymptotic theory \citep[cf.][Chapter~8]{vanderVaart07} shows that a maximum likelihood estimator $\hat{\theta}_n^{(mle)}$ based on $Z_1,\dots, Z_n$ achieves an asymptotic normal distribution 
\begin{equation}
\sqrt{n}(\hat{\theta}_n^{(mle)}-\theta) \xrightarrow[n\to\infty]{d} N(0, I_\theta(Q\P)^{-1})
\end{equation}
where the inverse of the Fisher-Information 
\begin{equation}
I_\theta(Q\P):=\E_{Z \sim QP_\theta}\left[\left(\frac{\partial \log QP_\theta(Z)}{\partial \theta}\right)^2\right]
\end{equation}
is the smallest variance achievable by regular estimators \citep[][Theorem~8.8]{vanderVaart07}. Thus, to minimize asymptotic variance, we should solve the optimization problem
\begin{equation}\label{eq:Opt}
\max_Q I_\theta(Q\P)\quad\text{subject to}\quad Q\in\mathcal Q_\epsilon.
\end{equation}
This is a very challenging optimization problem because the set of Markov kernels $\mathcal Q_\epsilon$ is a huge, non-parametric set and we have to even search over all possible output spaces $(\mathcal Z,\mathcal G)$ from which $Q$ can draw $Z_i$. Furthermore, it turns out that for fixed output space $(\mathcal Z,\mathcal G)$ the mapping $Q\mapsto I_\theta(Q\P)$ is convex, but we are trying to maximize this objective function. Hence, we may be confronted with numerous local and global optima at the boundary of the feasible set $\mathcal Q_\epsilon$. \citet{Steinberger24} provided a numerical scheme for approximately solving \eqref{eq:Opt} based on a linear program representation of \citet{Kairouz16}. However, the runtime of the linear program scales exponentially in the accuracy of approximation. \citet{Steinberger24} also showed that indeed, $[\sup_{Q\in\mathcal Q_\epsilon} I_\theta(Q\P)]^{-1}$ is the smallest possible asymptotic variance among all (sequentially interactive) locally differentially private and regular estimators. Here, regular means that the scaled estimation error $\sqrt{n}(\hat{\theta}_n-\theta_n)$ converges in distribution to a limiting distribution along the sequence of true parameters $\theta_n=\theta+h/\sqrt{n}$ for all $h \in \R$ and such that the limiting distribution does not depend on $h$ (see \citet[Section~8.5]{vanderVaart07} and Theorem~\ref{thm:efficiency} below). Regularity is indeed necessary to establish the mentioned lower bound on the asymptotic variance. Without regularity one could still show that the lower bound holds at least for all $\theta\in\R$ outside of a Lebesgue null set (see \citet[Theorem~8.8]{vanderVaart07} for the classical case and the discussion at the end of Section~3 in \citet{Steinberger24} for the locally private case), but certain pathologies may arise for irregular estimators \citep[see][Example~2.5 and Example~2.7]{Lehmann98}. Fortunately, we can show in Theorem~\ref{thm:efficiency} below that our proposed locally private estimator is also regular and achieves the asymptotic lower bound.

Our main contribution is to show that the following simple \textbf{sign mechanism} is an exact solution of \eqref{eq:Opt} in the case where the privacy parameter $\epsilon$ is sufficiently small: Given that $X_i=x$, we generate $Z_i$ by

\begin{equation}
\label{def:Z_i}
Z_i = \begin{cases}
    \sgn(x - \theta) \hspace{0.2cm} \text{with probability}\; p_\epsilon \\
     -\sgn(x - \theta) \hspace{0.2cm} \text{with probability}\; 1-p_\epsilon,
\end{cases} 
\end{equation}
where $p_\epsilon := \frac{e^\epsilon}{1 + e^\epsilon}$ and $\sgn(0) := 1$. In other words, for $z\in \mathcal Z = \{-1,1\}$ and $x\in\R$, we set 
\begin{equation}\label{eq:sgnMech}
\begin{aligned}
Q_{\theta,\epsilon}^{\text{sgn}}(\{z\}|x) :&= P(Z_i=z|X_i=x)\\ 
&= 
\begin{cases}
p_\epsilon, &\text{if } z=\sgn(x-\theta),\\
1-p_\epsilon, &\text{if }z\neq\sgn(x-\theta).
\end{cases}
\end{aligned}
\end{equation} 
Notice that this mechanism can also be represented as a composition $Q_{\theta,\epsilon}^{\text{sgn}} = Q^{RR}_\epsilon T_\theta$, where $T_\theta(x) := \sgn(x-\theta)$ and $Q^{RR}_\epsilon$ is the randomized response mechanism of \citet{Warner65} that flips the sign of the input with probability $1-p_\epsilon$. In particular, we easily conclude that $Q_{\theta,\epsilon}^{\text{sgn}} \in\mathcal Q_\epsilon$. We can now state and prove the following result.

\begin{theorem}\label{thm:main}
If $\epsilon\le 1.04$ then
\begin{equation}
I_\theta(Q\P) \le I_\theta(Q_{\theta,\epsilon}^{\text{sgn}}\P) = \frac{2}{\pi} \frac{(e^\epsilon - 1)^2}{(e^\epsilon + 1)^2},     
\end{equation}
for all $\theta\in\R$ and all $Q\in\mathcal Q_\epsilon$. In particular, the sign mechanism solves \eqref{eq:Opt}.
\end{theorem}

Our proof is based on the idea of quantizing the normal distribution and solving the problem of maximizing the private Fisher Information for discrete distributions simultaneously for all quantization levels. We apply results from \cite{Kairouz16} and a delicate duality argument to solve the discrete case. The main steps of the proof are collected in Subsection~\ref{sec:proof-main}.

An obvious issue with the sign mechanism in \eqref{eq:sgnMech} is that it depends on the unknown parameter $\theta$ which we want to estimate. By incorporating the sign mechanism into a two-stage scheme as in \citet{Steinberger24} we end up with the following procedure for asymptotically efficient parameter estimation: Split the agents into two groups of size $n_1$ and $n_2:=n - n_1$ such that $\lim \frac{n_1}{n} = 0$ and $n_1(n) \to \infty$.\footnote{We parametrize $n_1=n_1(n)$ by $n$ such that all limits are understood as $n\to\infty$.} Use the first group to get a consistent but not necessarily efficient estimate of the unknown mean $\theta$ and the second group to update the first-stage estimate to make it asymptotically efficient. The two stages of the procedure are described below, where $\theta_0$ is an arbitrary initial guess for the unknown mean $\theta$.

\begin{enumerate}
    \item 
    \begin{itemize}
    \item Apply the sign mechanism $Q^{\text{sgn}}_{\theta_0,\epsilon}$ at the initial value $\theta_0$ to $X_1,\dots, X_{n_1}$ to obtain iid sanitized data $Z_i\thicksim Q_{\theta_0,\epsilon}^{\text{sgn}}P_\theta$, $i=1,\dots, n_1$. 
    
    \item Compute 
    \begin{equation}
    \tilde{\theta}_{n_1} = 
    \theta_0 - \Phi^{-1}\left(\frac12 - \frac12\frac{e^\epsilon + 1}{e^\epsilon - 1}\bar{Z}_{n_1}\right), 
    \label{eq:est1}
    \end{equation}
    if   $|\bar{Z}_{n_1}|<\frac{e^\epsilon - 1}{e^\epsilon + 1}$, otherwise $\tilde{\theta}_{n_1} =\theta_0$,
    where $\Phi$ is the cumulative distribution function of the standard normal distribution and $\bar{Z}_{n_1} = \frac{1}{n_1}\sum\limits_{i = 1}^{n_1}Z_i$. 
    \end{itemize}
    \item 
    \begin{itemize}
    \item Apply the sign mechanism $Q^{\text{sgn}}_{\tilde{\theta}_{n_1},\epsilon}$ at the first stage estimate $\tilde{\theta}_{n_1}$ to the data $X_{n_1+1},\dots, X_n$ in the second group to generate values $Z_i \thicksim Q^{\text{sgn}}_{\tilde{\theta}_{n_1},\epsilon}P_\theta$, $i=n_1+1,\dots, n$.
    
     \item Compute 
    \begin{equation}
    \hat{\theta}_{n} = 
    \tilde{\theta}_{n_1} - \Phi^{-1}\left(\frac12 - \frac12\frac{e^\epsilon + 1}{e^\epsilon - 1}\bar{Z}_{n_2}\right),
    \label{eq:est2}
    \end{equation}
if   $|\bar{Z}_{n_2}|<\frac{e^\epsilon - 1}{e^\epsilon + 1}$, otherwise $\hat{\theta}_{n} = \tilde{\theta}_{n_1}$
where $\bar{Z}_{n_2} = \frac{1}{n_2}\sum\limits_{i = n_1+1}^{n}Z_i$.
\end{itemize}
\end{enumerate}

Relying on our main Theorem~\ref{thm:main}, we show that the above two-stage estimation procedure is regular and asymptotically achieves minimal variance. Let $R_\theta$ denote the distribution of the full sanitized data $Z_1,\dots, Z_n$ when the true unknown mean of the $X_i$ is $\theta\in\R$. Thus, $R_{\theta+h/\sqrt{n}}$ is the distribution of the sanitized data when the true unknown mean of the $X_i$ is $\theta_n:= \theta + h/\sqrt{n}$. Hence, on top of asymptotic normality with optimal asymptotic variance (i.e., the case $h=0$), the following result also establishes regularity of the locally private estimator sequence $\hat{\theta}_n$. Please refer to Appendix \ref{sec:prof-efficiency} for the proof.

\begin{theorem}\label{thm:efficiency}
If $\epsilon\le 1.04$ then the two-stage locally private estimation procedure described above satisfies
\begin{equation*}\sqrt{n} \left(\hat{\theta}_n-\theta-\frac{h}{\sqrt{n}}\right) \xrightarrow[n\to\infty]{R_{\theta+\frac{h}{\sqrt{n}}}} N\left(0, \left[\sup_{Q \in \mathcal{Q}_\epsilon} I_\theta(Q\P)\right]^{-1}\right),
\end{equation*}
 for any value $\theta_0$, $h$ and $n_1 = o(n)$ such that $n_1(n) \to \infty$.
\end{theorem}

Notice, however, that the two-stage estimator $\hat{\theta}_n$ for a finite $n$ depends on the choice of initial value $\theta_0$ and on the size $n_1$ of the first group. Hence, $\theta_0$ and $n_1$ are tuning parameters in our estimation procedure. In Section~\ref{sec:sim}, we investigate the impact of $\theta_0$ and $n_1$ on the performance of our locally private estimator in a finite sample simulation study.

\subsection{Related Literature}

Local differential privacy originated about 20 years ago \citep[see][]{Dinur03, Dwork04, Dwork06, Dwork08a, Evfim03} and has since become an incredibly popular way of data privacy protection when there is no trusted third party available. Despite its popularity, some of the most basic learning problems, such as the one considered here, have not been fully solved under LDP. Much attention has recently been paid to discrete distribution estimation and deriving upper bounds on the estimation error. For instance, \cite{Wang16} have suggested an optimal mechanism for mutual information maximization for discrete distributions under LDP. \cite{Ye18} are studying the locally private minimax risk for discrete distributions. \cite{Nam22} is perhaps closer to our work. They consider the problem of maximizing a Fisher-Information in the local privacy mechanism, but they restrict to the 1-bit communication constraint which, for our purpose, is an oversimplified case. \citet{Barnes20} also investigate and bound Fisher-Information when original data are perturbed using a differentially private randomization mechanism. But they have not attempted to find an optimal randomization mechanism or an efficient estimation procedure. Gaussian mean estimation is considered by \cite{Jos19} and they also use a two-stage procedure for parameter estimation. They obtain high-probability bounds for the estimation error, which is of order $n^{-1/2}$ for the two-stage protocol, but unlike our work, they do not obtain the optimal asymptotic constant. \citet{Duchi19} proposed a one-stage algorithm for Gaussian mean estimation, which resembles our first-stage estimator. While their method can outperform a two-stage protocol, this advantage holds primarily when the initial guess  $\theta_0$ is close to the true parameter $\theta$ and the sample size is relatively small. \citet{Asi22} introduced the PrivUnitG algorithm, designed specifically for high-dimensional problems where the data resides on the unit sphere. While their method is effective in these settings, it is worth noting that for one-dimensional data, where the values are 
$\pm 1$, their approach differs from the randomized response, which is provably optimal for this problem. Consequently, their method is inefficient for one-dimensional cases. Building on these insights, \citet{Asi23} developed the ProjUnit framework, which further optimizes mean estimation by projecting data into lower-dimensional subspaces before applying an optimal randomizer. A recent work by \citet{liu2023online} introduced a continual, gradient-like procedure to iteratively refine the estimate of $\theta$. They have shown that the procedure achieves a variance that is optimal among binary mechanisms. However, we will show that this variance is asymptotically optimal among all LDP mechanisms.

The conceptual foundation to the problem of finding the optimal Fisher-Information mechanism was laid by \citet{Steinberger24}. However, while in that work the optimal mechanism and locally private estimation procedure have to be approximated by a computationally expensive numerical optimization routine, we here derive an exact and simple closed-form privacy mechanism and estimator.
\section{Proof of main results}
\label{sec:main}

\subsection{Proof of Theorem~\ref{thm:main}}
\label{sec:proof-main}
In this section, we prove Theorem~\ref{thm:main}. We first follow the approach of \citet{Steinberger24} and quantize the Fisher-Information to then solve a discrete version of \eqref{eq:Opt}. However, while \citet{Steinberger24} had to rely on a numerical algorithm for the discrete optimization, which has exponential runtime in the resolution of the discretization, we here solve all the discrete problems exactly by providing a closed form solution and show that it is the same for all resolution levels $k$ of the approximation. We then conclude that the solution to the simplified discrete problems is also the global solution. Notice that all the regularity conditions imposed by \citet{Steinberger24} are satisfied for the Gaussian location model $\P = \{N(\theta,1):\theta\in\R\}$ that we consider here \citep[cf. Section~5.2 in][]{Steinberger24}. In particular, the Fisher-Information $I_\theta(Q\P)$ is well defined and finite for any privacy mechanism $Q\in\mathcal Q_\epsilon$.

\subsubsection{Discrete approximation}

We begin by defining what is called a consistent quantizer in \citet[cf. Definition~3 and Section~5 in that reference]{Steinberger24}. For an even positive integer $k$ and for $j=1,\dots, k-1$, let $B_j := (x_{j-1}, x_j]$ and $B_k := (x_{k-1}, \infty)$, where $x_j:=\Phi^{-1}(j/k)$ and $\Phi$ is the cdf of the standard normal distribution. Now set $T_{k,\theta}(x) := \sum_{j=1}^k j\mathds 1_{B_j}(x-\theta)$. Notice that $T_{k,\theta}$ maps the real line $\R$ into the discrete set $[k]:=\{1,\dots, k\}$, hence, it is called a quantizer. We write $T_{k,\theta_0}\P := \{r_{\theta,\theta_0}:\theta\in\R\}$ for the resulting quantized model, where $r_{\theta,\theta_0}(j) := P_\theta( T_{k,\theta_0}^{-1}(\{j\}) ) = P_{\theta}(B_j+\theta_0) = \Phi(x_j+\theta_0-\theta) - \Phi(x_{j-1}+\theta_0-\theta)$ is the probability mass function of the quantized data $T_{k,\theta_0}(X_i)$ and we set $\Phi(-\infty):=0$, $\Phi(\infty):=1$, $\phi(x) := \Phi'(x)$ and $\phi(\pm\infty):=0$. Lemmas~4.7, 4.9, 5.1 and 5.2 in \citet{Steinberger24} show that for every even $k$ and every $\theta\in\R$
\begin{align}\label{eq:quantization}
\sup_{Q\in\mathcal Q_\epsilon(\R)} I_\theta(Q\P) &\le 
\sup_{Q\in\mathcal Q_\epsilon([k])} I_\theta(QT_{k,\theta}\P) + \Delta_k
\\
&=\sup_{Q\in\mathcal Q_\epsilon([k]\to[k])} I_\theta(QT_{k,\theta}\P) + \Delta_k,
\end{align}
where $\Delta_k\to0$ as $k\to\infty$, $\mathcal Q_\epsilon([k])$ is the set of $\epsilon$-private mechanisms that take inputs from $[k]$ and produce random outputs from some arbitrary measurable space $(\mathcal Z,\mathcal G)$, and $\mathcal Q_\epsilon([k]\to[k])$ are the $\epsilon$-private mechanisms whose inputs and outputs both take values in $[k]$. Thus, the elements of $\mathcal Q_\epsilon([k]\to[k])$ can be represented as $k\times k$ column stochastic matrices with the property that $Q_{ij}\le e^\epsilon Q_{ij'}$, for all $i, j, j'\in[k]$. Since obviously $\sup_{Q\in\mathcal Q_\epsilon([k]\to[k])} I_\theta(QT_{k,\theta}\P)\le \sup_{Q\in\mathcal Q_\epsilon} I_\theta(Q\P)$, \eqref{eq:quantization} yields
$$
\sup_{Q\in\mathcal Q_\epsilon([k]\to[k])} I_\theta(QT_{k,\theta}\P) \to \sup_{Q\in\mathcal Q_\epsilon} I_\theta(Q\P), \quad\forall\theta\in\R,
$$
as $k\to\infty$, $k$ even. In the remainder of the proof, we will show that $\sup_{Q\in\mathcal Q_\epsilon([k]\to[k])} I_\theta(QT_{k,\theta}\P)=I_\theta(Q_{\theta,\epsilon}^{\text{sgn}}\P)$ for every even $k$. Thus, the sign mechanism in \eqref{eq:sgnMech} must be a solution of \eqref{eq:Opt}.

\subsubsection{Evaluating the quantized objective function}

Next, we provide an explicit expression for the quantized private Fisher-Information $I_\theta(QT_{k,\theta_0}\P)$, for arbitrary $Q\in \mathcal Q_\epsilon([k]\to[k])$. The quantized and $Q$-privatized model $QT_{k,\theta_0}\P = \{m_{\theta,\theta_0}:\theta\in\R\}$ can be expressed via its probability mass functions $m_{\theta,\theta_0}(i) := \sum_{j=1}^k Q_{ij}r_{\theta,\theta_0}(j)$, $i\in[k]$. Thus, we have
\begin{align*}
I_\theta(QT_{k,\theta_0}\P) &= \E_{Z\thicksim QT_{k,\theta_0}P_\theta}\left[\left( \frac{\partial \log m_{\theta,\theta_0}(Z)}{\partial \theta}\right)^2 \right]\\
&=
\sum_{i=1}^k \frac{\dot{m}_{\theta,\theta_0}(i)^2}{m_{\theta,\theta_0}(i)},
\end{align*}
where $\dot{m}_{\theta,\theta_0}(i) := \frac{\partial}{\partial\theta}m_{\theta,\theta_0}(i) = \sum_{j=1}^k Q_{ij}[\phi(x_{j-1}+\theta_0-\theta) - \phi(x_{j}+\theta_0-\theta)]$ and the ratio is to be understood as equal to zero if the denominator $m_{\theta,\theta_0}(i)$ is zero. Abbreviating $y_j:= \phi(x_{j-1}) - \phi(x_{j})$, we arrive at
\begin{equation}\label{eq:FI}
I_\theta(QT_{k,\theta}\P) = \sum_{i=1}^k \frac{\left( \sum_{j=1}^k Q_{ij}y_j \right)^2}{\frac1k \sum_{j=1}^k Q_{ij}} = \sum_{i=1}^k \mu(Q_{i\cdot}^T),
\end{equation}
for 
$
\mu(v) := k\frac{(v^Ty)^2}{v^T\vec{1}},\quad v\in\mathcal C_k := \{u\in\R_+^k:u_j\le e^\epsilon u_{j'}, \forall j,j'\in[k]\},
$
and $\mu(0):=0$. In particular, we see that $I_\theta(QT_{k,\theta}\P)=I_0(QT_{k,0}\P)$ for all $\theta\in\R$. For later use we note that $y_j = - y_{k-j+1}$, $\sum_{j=1}^{k/2} y_j = -\phi(0) = -\sum_{j=k/2+1}^{k} y_j$ and $\sum_{j=1}^{k} y_j=0$.

We conclude this subsection by computing the Fisher-Information for the sign mechanism in \eqref{eq:sgnMech}. Use the fact that $Q_{\theta,\epsilon}^{\text{sgn}} = Q_\epsilon^{RR}\circ T_{2,\theta}\in\mathcal Q_\epsilon([2]\to[2])$, where the randomized response mechanism can be represented in matrix form as
\begin{equation}
Q_\epsilon^{RR} = \begin{pmatrix}
\frac{e^\epsilon}{1+e^\epsilon} & \frac{1}{1+e^\epsilon}\\
\frac{1}{1+e^\epsilon} & \frac{e^\epsilon}{1+e^\epsilon}
\end{pmatrix}.
\end{equation}
Thus, 
\begin{align*}
I_\theta(Q_{\theta,\epsilon}^{\text{sgn}}\P) &= I_\theta(Q_\epsilon^{RR}T_{2,\theta}\P) = \sum_{i=1}^2 \frac{\left( \sum_{j=1}^2 [Q_\epsilon^{RR}]_{ij}y_j \right)^2}{\frac12 \sum_{j=1}^2 [Q_\epsilon^{RR}]_{ij}} \\
&=
\frac{1}{1+e^\epsilon} \left( \frac{ \phi(0)^2(1-e^\epsilon)^2}{\frac12 (1+e^\epsilon)}
+ 
\frac{ \phi(0)^2(e^\epsilon-1)^2}{\frac12 (1+e^\epsilon)}
\right)\\
&=
\frac{2}{\pi} \frac{(e^\epsilon-1)^2}{(e^\epsilon+1)^2}.
\end{align*}

\subsubsection{Reformulation as a Linear Program}

From Lemma~4.5 of \citet{Steinberger24}, $Q\mapsto I_0(QT_{k,0}\P)$ is a continuous, sublinear and convex function on the compact set $\mathcal Q_\epsilon([k]\to[k])\subseteq\R^{k\times k}$. Therefore, a maximizer exists and Theorems~2 and 4 of \cite{Kairouz16} allow us to conclude that the maximization problem 
\begin{align}\label{eq:approx}
\max_Q I_0(QT_{k,0}\P)\quad\text{subject to}\quad Q\in\mathcal Q_\epsilon([k]\to[k])
\end{align}
has the same optimal value as the linear program
\begin{equation}
\begin{aligned}
\max_{\alpha \in \mathbb{R}^{2^k}} \quad & \sum\limits_{j = 1}^{2^k} \mu(S^{(k)}_{\cdot j}) \alpha_j \\
\textrm{s.t.} \quad & S^{(k)}\alpha = \vec{1} \\
  &\alpha\geq0,    \\
\end{aligned}\label{eq:LP}
\end{equation}
where $S^{(k)}$ is a staircase matrix defined as follows: For $ 0 \le j \le 2^k - 1$ consider its binary representation $b_{j}\in\{0,1\}^k$ with $j = \sum_{i = 1}^{k}b_{ij}2^{k - i}$. Then $(S^{(k)})_{i,j+1} := b_{ij} (e^\epsilon - 1) + 1$. For instance, with $k = 3$ we get the following matrix:

\setcounter{MaxMatrixCols}{16}

$$
S^{(3)} = \begin{bmatrix}
1 & 1 & 1 & 1 &  e^\epsilon & e^\epsilon & e^\epsilon & e^\epsilon\\
1 & 1  &e^\epsilon & e^\epsilon & 1 & 1 &e^\epsilon & e^\epsilon \\
1 &  e^\epsilon &  1 & e^\epsilon & 1 & e^\epsilon & 1 & e^\epsilon\\ 
\end{bmatrix}.
$$
Moreover, if $\alpha^*$ is a solution of \eqref{eq:LP}, then $Q^*= [S^{(k)}\text{diag}(\alpha^*)]^T\in\R^{2^k\times k}$ is a solution of \eqref{eq:approx}. From \citet[][Theorem~2]{Kairouz16} we know that an optimal mechanism $Q^*$ has at most $k$ non-zero rows. Since zero rows do not contribute to the Fisher-Information \eqref{eq:FI}, we can remove them from $Q^*$ to obtain $Q^*\in\R^{k\times k}$. 

Finally, notice that $\alpha_j^*  := (1+e^\epsilon)^{-1}$ for $j\in\{2^{k/2},2^k-2^{k/2}+1\}$ and $\alpha_j^* := 0$ else, is a feasible point of \eqref{eq:LP}, because $b_{2^{k/2}} = \vec{1} - b_{2^k-2^{k/2}+1}$ and thus
\begin{align*}
S^{(k)}\alpha^* &= \frac{1}{1+e^\epsilon}\left[(b_{2^{k/2}} + b_{2^k-2^{k/2}+1})(e^\epsilon-1) + 2 \cdot \vec{1} \right]\\
&= \vec{1}.
\end{align*}
Furthermore, $\alpha^*$ achieves an objective function value of
\begin{align*}
&\frac{1}{1+e^\epsilon}\left[\mu(b_{2^{k/2}}(e^\epsilon-1) + \vec{1} )+\mu(b_{2^k-2^{k/2}+1}(e^\epsilon-1) + \vec{1})\right] \\
&\quad=\frac{2}{1+e^\epsilon}\left[\frac{\phi(0)^2(e^\epsilon-1)^2}{e^\epsilon+1} + \frac{\phi(0)^2(e^\epsilon-1)^2}{e^\epsilon+1}\right]\\
&\quad=I_\theta(Q_{\theta,\epsilon}^{\text{sgn}}\P).
\end{align*}
Hence, we have $\sup_{Q\in\mathcal Q_\epsilon([k]\to[k])} I_\theta(QT_{k,\theta}\P) = \sup_{Q\in\mathcal Q_\epsilon([k]\to[k])} I_0(QT_{k,0}\P)\ge I_\theta(Q_{\theta,\epsilon}^{\text{sgn}}\P)$, for every even $k>0$. It remains to establish an upper bound, which we do by a duality argument.

\subsubsection{Guessing a dual solution}

In this section, we study the dual of the LP in \eqref{eq:LP}. By weak duality, any feasible value of the dual provides an upper bound on the objective function of the primal problem. Hence, the challenge is to identify a feasible value of the dual that achieves an objective function equal to $I_\theta(Q_{\theta,\epsilon}^{\text{sgn}}\P)$. The dual program reads as follows:

\begin{equation}\label{eq:dual}
\begin{aligned}
\min_{\beta \in \mathbb{R}^k} \quad&\vec{1}^T \beta\\
\textrm{s.t.} \quad &(S^{(k)})^T\beta \ge \vec{\mu}
\end{aligned}
\end{equation}
where $\vec{\mu}:= (\mu_j)_{j=1}^{2^k} := (\mu(S^{(k)}_{\cdot j}))_{j=1}^{2^k}\in\R^{2^k}$. The proof of Theorem~\ref{thm:main} is finished if we can identify a feasible point $\beta\in\R^k$ of \eqref{eq:dual} satisfying $\vec{1}^T \beta=\sum_{j=1}^k \beta_j = \frac{2}{\pi} t_\epsilon^2$, where $t_\epsilon:= \frac{e^\epsilon-1}{e^\epsilon+1}$. Our guess for such a $\beta^*\in\R^k$ is
\begin{equation}
\beta_j^* := -\frac{2t_\epsilon^2}{\pi k} + |y_j|t_\epsilon^2 \sqrt{\frac{8}{\pi}}.
\end{equation}
This clearly satisfies $\vec{1}^T \beta^*= \frac{2}{\pi} t_\epsilon^2$. It is the main technical challenge of the proof to show that $\beta^*$ is a feasible point of \eqref{eq:dual}. 

For fixed $j\in[2^k]$, we have to verify $(S_{\cdot j}^{(k)})^T\beta^* \ge \mu(S_{\cdot j}^{(k)})$. It will be convenient to partition the index set $[k]$ as follows
$[k] =  I_{1}^0 \cup I_{e^\epsilon}^{0} \cup I_{1}^1 \cup I_{e^\epsilon}^{1}$ where $I_{1}^0 = \{i|\; i \le \frac{k}{2}, S_{i,j}^{(k)} = 1\}$, $I_{e^\epsilon}^0 = \{i|\; i \le \frac{k}{2}, S_{i,j}^{(k)} = e^\epsilon\}$ and $I_{1}^{1}, I_{e^\epsilon}^{1}$ defined similarly for the second half of indices $i=\frac{k}{2}+1,\dots, k$. Then we can rewrite the required inequality $(S_{\cdot j}^{(k)})^T\beta^* \ge \mu(S_{\cdot j}^{(k)})$ as follows:

\begin{align*}
&\sum\limits_{i \in I_{1}^0} \beta_i^* + e^\epsilon\sum\limits_{i \in I_{e^\epsilon}^0} \beta_i^* + \sum\limits_{i \in I_{1}^1} \beta_i^* + e^\epsilon\sum\limits_{i \in I_{e^\epsilon}^1} \beta_i^*  \\
&\ge  \frac{\left(\sum\limits_{i \in I_{1}^0} y_i + e^\epsilon\sum\limits_{i \in I_{e^\epsilon}^0} y_i + \sum\limits_{i \in I_{1}^1} y_i + e^\epsilon\sum\limits_{i \in I_{e^\epsilon}^1} y_i\right)^2}{\frac{1}{k}\left(|I_1^0| + |I_1^1| + e^\epsilon|I_{e^\epsilon}^0| + e^\epsilon|I_{e^\epsilon}^1|\right)}.
\end{align*}
For convenience let us denote $m_1 = |I_1^0|$, $m_2 = |I_{e^\epsilon}^1|$, $a_1 = \sqrt{2\pi}\sum\limits_{i \in I_1^0} |y_i|$, $a_2 = \sqrt{2\pi} \sum\limits_{i \in I_{e^\epsilon}^1} |y_i|$ such that $m_1, m_2 \le \frac{k}{2}$ and $a_1, a_2 \in [0, 1]$. With this new notation, we can rewrite the sum on the left-hand-side as
\begin{equation}\label{eq:LHS}
(1 - e^\epsilon) \sum\limits_{i \in I_{1}^0} \beta_i^* + (e^\epsilon - 1) \sum\limits_{i \in I_{e^\epsilon}^1} \beta_i^* + e^\epsilon\sum\limits_{i = 1}^{k/2} \beta_i^* + \sum\limits_{i = k/2 + 1}^{k} \beta_i^* 
\end{equation}
and the ratio on the right-hand-side as
\begin{equation}\label{eq:RHS}
\frac{\left(\frac{e^\epsilon - 1}{\sqrt{2\pi}}a_1 + \frac{e^\epsilon - 1}{\sqrt{2\pi}} a_2 - \frac{ e^\epsilon }{\sqrt{2\pi}} + \frac{ 1 }{\sqrt{2\pi}}\right)^2}{\frac{1}{k}\left(k/2 + m_1 - m_2 + e^\epsilon(k/2 + m_2 - m_1)\right)}.
\end{equation}
Substituting the values for $\beta^*$ we simplify the expression in \eqref{eq:LHS} to obtain

\begin{align*}
&(1 - e^\epsilon) \sum\limits_{i \in I_{1}^0} \beta_i^* + (e^\epsilon - 1) \sum\limits_{i \in I_{e^\epsilon}^1} \beta_i^* + e^\epsilon\sum\limits_{i = 1}^{k/2} \beta_i^* + \sum\limits_{i = k/2 + 1}^{k} \beta_i^* \\
%&\quad= (1 - e^\epsilon) \sum\limits_{i \in I_{1}^0} \left[-\frac{2t_\epsilon^2}{\pi k} + |y_i| t_\epsilon^2 \sqrt{\frac{8}{\pi}}\right] + (e^\epsilon - 1)\sum\limits_{i \in I_{e^\epsilon}^1} \left[-\frac{2t_\epsilon^2}{\pi k} + |y_i| t_\epsilon^2 \sqrt{\frac{8}{\pi}}\right] + (1 + e^\epsilon) \frac{t_\epsilon^2}{\pi} \\
&\quad= (1 - e^\epsilon)\frac{2t_\epsilon^2}{\pi} \left(-\frac{m_1}{k} + a_1\right)\\
&\quad\quad+ (e^\epsilon - 1) \frac{2t_\epsilon^2}{\pi} \left(-\frac{m_2}{k} + a_2\right) + (1 + e^\epsilon) \frac{t_\epsilon^2}{\pi}  \\
&\quad= (e^\epsilon - 1)\frac{2t_\epsilon^2}{\pi}\left[\frac{m_1 - m_2}{k} + a_2 - a_1\right] + (1 + e^\epsilon) \frac{t_\epsilon^2}{\pi}.
\end{align*}
Next, we simplify \eqref{eq:RHS} to get 

\begin{align*}
    &\frac{\left(\frac{e^\epsilon - 1}{\sqrt{2\pi}}a_1 + \frac{e^\epsilon - 1}{\sqrt{2\pi}} a_2 - \frac{ e^\epsilon }{\sqrt{2\pi}} + \frac{ 1 }{\sqrt{2\pi}}\right)^2}{\frac{1}{k}\left(\frac{k}{2} + m_1 - m_2 + e^\epsilon(\frac{k}{2} + m_2 - m_1)\right)}\\
&\quad= \frac{\left(a_1 + a_2 - 1\right)^2(e^\epsilon - 1)^2}{\frac{2\pi}{k}\left(\frac{k}{2}(1 + e^\epsilon) + (m_2 - m_1)(e^\epsilon - 1)\right)} \\
&\quad=\frac{(a_1 + a_2 - 1)^2(e^\epsilon - 1)^2}{\pi(1 + e^\epsilon) \left(1 + 2t_\epsilon \frac{m_2 - m_1}{k}\right)}.
\end{align*}
After these simplifications, we see that the inequality we need to verify is given by
\begin{align*}
    &(e^\epsilon - 1)\frac{2t_\epsilon^2}{\pi}\left[\frac{m_1 - m_2}{k} + a_2 - a_1\right] + (1 + e^\epsilon) \frac{t_\epsilon^2}{\pi}\\  
    &\quad\ge \frac{(a_1 + a_2 - 1)^2(e^\epsilon - 1)^2}{\pi(1 + e^\epsilon) \left(1 + 2t_\epsilon \frac{m_2 - m_1}{k}\right)}.
\end{align*}
Dividing by $(1 + e^\epsilon)t_\epsilon^2 / \pi$, this can be further simplified to
\begin{equation}
\begin{aligned}\label{eq:finalIneq}
&\mathcal{L}\left(a_1, a_2, \frac{m_1}{k}, \frac{m_2}{k}\right) := (a_2 - a_1)\left[2t_\epsilon  + 4t_\epsilon^2 \frac{m_2 - m_1}{k}\right]\\
&\quad-4t_\epsilon^2 \left(\frac{m_2 - m_1}{k}\right)^2 - (a_1 + a_2 - 1)^2 + 1 \ge 0.
\end{aligned}
\end{equation}
Without loss of generality, we can assume $a_1 + a_2 \le 1$; otherwise, we could have chosen $a_1 = \sqrt{2\pi}\sum\limits_{i \in I^1_1} |y_i|$ and $a_2 =\sqrt{2\pi}\sum\limits_{i \in I_{e^\epsilon}^0} |y_i|$, maintaining the same inequality due to the symmetry of $\beta_i^* = \beta_{k + 1 - i}^*$ and antisymmetry in $y_i = -y_{k + 1 - i}$ , but with the sum $a_1 + a_2 \le 1$, because the total sum equals $\sqrt{2\pi}\sum\limits_{i = 1}^{k}|y_i| = 2$. Given this, we can prove the inequality $\mathcal{L}\left(a_1, a_2, \frac{m_1}{k}, \frac{m_2}{k}\right) \ge 0$ using the following four lemmas.

\begin{lemma}
For any $k$, the sequence $y_j = \phi(\Phi^{-1}((j - 1)/k)) - \phi(\Phi^{-1}(j/k))$, $j=1,\dots, k$, is increasing.
\end{lemma}
\begin{proof}
For $x\in(1, k]$, define $g(x) := \phi(\Phi^{-1}((x - 1)/k)) - \phi(\Phi^{-1}(x/k))$ and $g(1):=-\phi(\Phi^{-1}(1/k)$. Since $g$ is continuous on $[1, k]$, it suffices to show that $g$ is strictly increasing on $(1, k)$. Using the fact that $\phi'(x) = (-x)\phi(x)$ and the inverse function theorem we get
\begin{align*}
g'(x) &= \frac1k\left( \frac{\phi'}{\phi}\circ\Phi^{-1}((x-1)/k)-\frac{\phi'}{\phi}\circ\Phi^{-1}(x/k)\right)\\
&=
\frac1k\left( \Phi^{-1}(x/k)-\Phi^{-1}((x-1)/k)\right)>0.
\end{align*}
Thus, $g$ is strictly increasing on $(1,k)$.
\end{proof}

\begin{lemma}
For $x \in [0, \frac{1}{2}]$ we have
\begin{equation}
\phi(0) - \phi\left(\Phi^{-1}\left(\frac{1}{2} + x\right)\right)   \ge \sqrt{\frac{\pi}{2}}x^2.
\end{equation}
\end{lemma}
\begin{proof}
For $x\in [0, \frac{1}{2}]$, define $g(x) := \phi(0) - \phi\left(\Phi^{-1}\left(\frac{1}{2} + x\right)\right) -    \sqrt{\frac{\pi}{2}}x^2$ then $g(0) = 0, g(\frac{1}{2}) = +\infty$. Since $g$ is continuous on $[0, \frac{1}{2}]$, it suffices to show that $g$ is strictly increasing on $(0, \frac{1}{2})$. Using the fact that $\phi'(x) = (-x)\phi(x)$ and the inverse function theorem we get

\begin{align*}
g'(x) &= -\frac{\phi'}{\phi}\circ\Phi^{-1}\left(\frac{1}{2} + x\right) - \sqrt{2\pi} x \\
&=  \Phi^{-1}\left(\frac{1}{2} + x\right) - \sqrt{2\pi} x,\\
g''(x) &= \frac{1}{\phi(\Phi^{-1}\left(\frac{1}{2} + x\right))} - \sqrt{2\pi} \ge 0.
\end{align*}
Since the first derivative is equal to zero at zero and non-decreasing, we see that the function $g$ itself must be non-decreasing.
\end{proof}
By combining these two lemmas, we can easily show 
\begin{align*}
    a_2 &= \sqrt{2\pi}\sum\limits_{i \in I_{e^\epsilon}^1}y_i \ge \sqrt{2\pi} \sum\limits_{i = k/2 + 1}^{k/2 + m_2}y_i\\
    &= \sqrt{2\pi}\left(\phi(0) - \phi\left(\Phi^{-1}\left(\frac{1}{2} + \frac{m_2}{k}\right)\right)\right) \ge \pi \frac{m_2^2}{k^2}.
\end{align*}
A similar lower bound holds for $a_1$ as well, namely $a_1 \ge \pi\frac{m_1^2}{k^2}$. Therefore, it is sufficient to prove the inequality \eqref{eq:finalIneq}  under restrictions that $a_1 + a_2 \le 1$, \; $a_1 \ge \pi \frac{m_1^2}{k^2}$, \; $a_2 \ge \pi \frac{m_2^2}{k^2}$, and  $m_1, m_2 \le \frac{k}{2}$. Consider the derivative of $\mathcal{L}\left(a_1, a_2, \frac{m_1}{k}, \frac{m_2}{k}\right)$ with respect to $a_2$,

\begin{equation}
2t_\epsilon \left(1 + 2 t_\epsilon \frac{m_2 - m_1}{k}\right) + 2(1 - (a_1 + a_2)) \ge 0.
\end{equation}
It is non-negative, therefore $\mathcal{L}\left(a_1, a_2, \frac{m_1}{k}, \frac{m_2}{k}\right)$ is an increasing function of  $a_2$ reaching its minimum at $a_2 = \pi \frac{m_2^2}{k^2}$. For $a_1$, we have a concave quadratic function that should reach minimum at one of the boundaries, which are $a_1 = \pi \frac{m_1^2}{k^2}$ and $a_1 = 1 - a_2 = 1 - \pi \frac{m_2^2}{k^2}$. Unfortunately, the minimum could occur at either point depending on the configuration. Therefore, we should consider two cases, which we formulate as the following lemmas, denoting $\frac{m_1}{k} =x$ and $\frac{m_2}{k} = y$.

\begin{lemma}
\label{lem:a_1_right}
For $ 0 \le x, y \le \frac{1}{2}$ and $\pi x^2 + \pi y^2 \le 1$ for the parameter $ 0 < t_\epsilon \le \frac{4\pi}{1 + 8\pi} \approx 0.4808$ we have the following inequality:
%$$2 t_\epsilon (2\pi y^2 - 1) + 4 t_\epsilon^2 (y - x) (2 \pi y^2 - 1) + 1 - 4 t_\epsilon^2 (y - x)^2 \ge 0.$$
\begin{equation*}
    \mathcal{L}\left(1 - \pi y^2, \pi y^2, x, y\right) \ge 0.
\end{equation*}
\end{lemma}
\begin{lemma}
\label{lem:a_1_left}
For $ 0 \le x, y \le \frac{1}{2}$ and $\pi x^2 + \pi y^2 \le 1$ for the parameter $0 < t_\epsilon \le \frac{1}{2}$ we have the following inequality:
%\begin{align*}
%2 t_\epsilon \pi (y^2 - x^2) + 4 t_\epsilon^2 \pi (y - x) (y^2 - x^2)\\
%- \pi^2 (x^2 + y^2)^2 + 2\pi (x^2 + y^2) - 4 t_\epsilon^2 (y - x)^2 \ge 0.
%\end{align*}
\begin{equation*}
    \mathcal{L}\left(\pi x^2, \pi y^2, x, y\right) \ge 0.
\end{equation*}
\end{lemma}
Combining these we conclude that for $t_\epsilon \le  \frac{4\pi}{1 + 8\pi}$ the statement of Theorem \ref{thm:main} holds. But this condition is satisfied for all $\epsilon \le \log(\frac{1+12\pi}{1+4\pi}) \approx 1.04822$, which concludes the proof.

\begin{figure*}
    \centering
%    \begin{subfigure}[b]{0.4\textwidth}
    \includegraphics[scale=0.36]{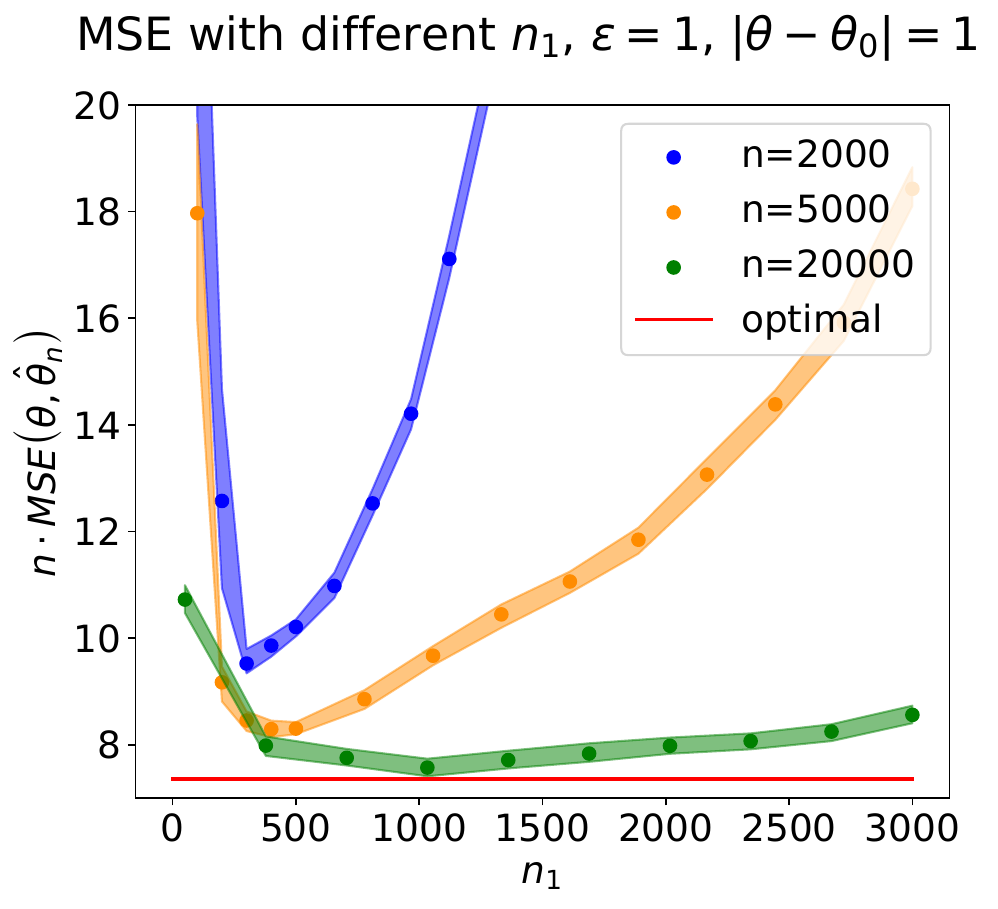}
%    \end{subfigure}
    %\hfill
%    \begin{subfigure}[b]{0.4\textwidth}
    \includegraphics[scale=0.36]{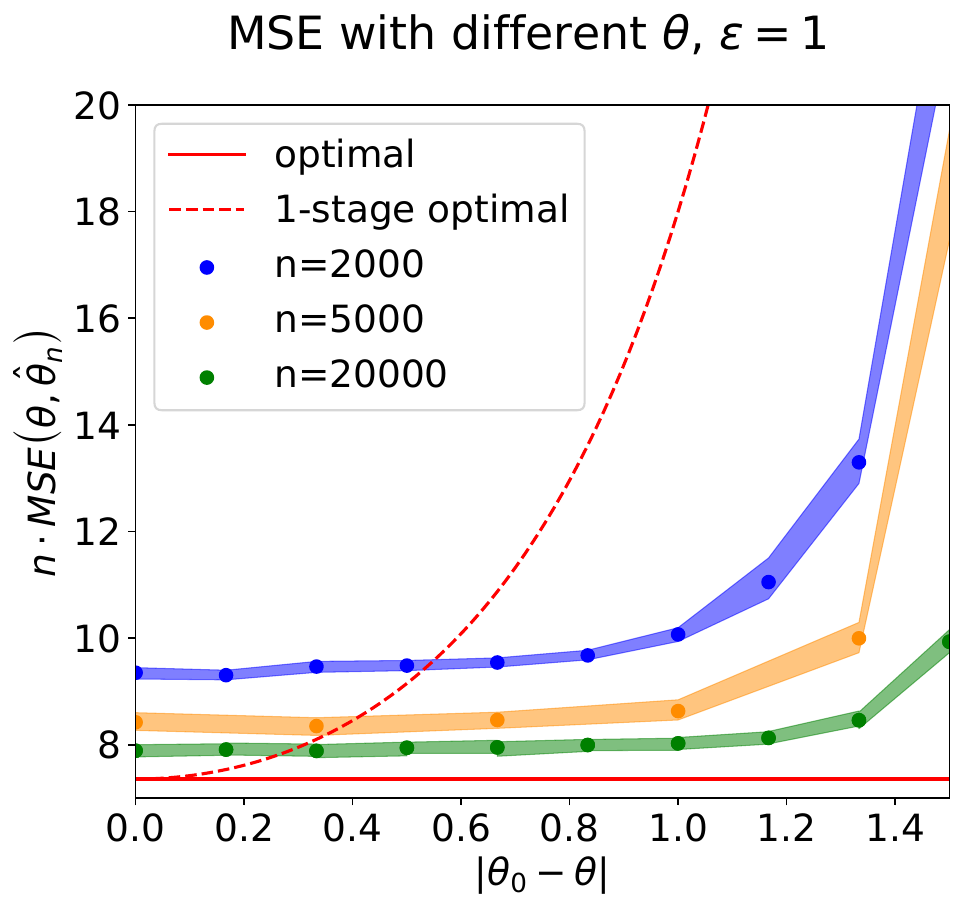}
%    \end{subfigure}
    
    \caption{Left panel: Scaled MSE of the private estimator as a function of the number $n_1$ of first-stage samples. 
    Right panel: Scaled MSE of the private estimator as a function of the initial guess $\theta_0$.}
    \label{fig:sim}
\end{figure*}
\subsection{One-Stage Variance}

In this subsection, we address the limitations of using a single-stage procedure. The asymptotic variance of an estimator based on only one stage as in \eqref{eq:est1} with $n_1=n$ can be expressed by the following formula:

$\label{eq:one_round_variance}
\frac{1}{4}\left(\frac{e^\epsilon + 1}{e^\epsilon - 1}\right)^2 \frac{1}{\phi (\theta - \theta_0)^2}\left(1 - \left(\frac{e^\epsilon - 1}{e^\epsilon + 1}\right)^2(1 - 2\Phi(\theta_0 - \theta))^2\right).$
This expression resembles the optimal asymptotic variance derived in Theorem \ref{thm:main} when $\theta_0 \approx \theta$. A crucial observation is that for single-stage estimation, the asymptotic variance deteriorates exponentially as $|\theta_0 - \theta|$ increases. This motivates the introduction of a two-stage procedure.

\begin{proof}
The estimator from the first stage, as described in \eqref{eq:est1}, is updated according to the following formula. We assume that all $n$ samples are used for this estimate, that is, we calculate the asymptotic variance of
\begin{equation}
    \tilde{\theta}_{n} = \theta_0 - \Phi^{-1}\left(\frac{1}{2} - \frac{1}{2}\frac{e^{\epsilon} + 1}{e^{\epsilon} - 1}\bar{Z}_n\right).
\end{equation}
Here, $\bar{Z}_n = \frac{1}{n}\sum_{i=1}^{n}Z_i$, where $Z_1, \dots, Z_n$ are iid random variables taking values in $\{-1, 1\}$ according to \eqref{def:Z_i} and $X_i \thicksim N(\theta,1)$. To calculate the asymptotic variance of $\tilde{\theta}_n = f(\bar{Z}_n)$ with $f(x) := \theta_0 - \Phi^{-1}\left(\frac12 - \frac 12 \frac{e^\epsilon + 1}{e^\epsilon -1} x \right)$, we apply the delta method,
\begin{equation*}
    \sqrt{n}(f(\bar{Z}_n) - f(\E(Z_1)) \xrightarrow[n\to\infty]{D} N(0, [f'(\E(Z_1))]^2\Var(Z_1)).
\end{equation*}
Thus, we need the expectation and variance of $Z_1$, which are given by
\begin{align*}
    &\mathbb{E}(Z_1) = \frac{e^\epsilon-1}{e^\epsilon+1}(1-2\Phi(\theta_0-\theta)),\\
    &\Var(Z_1) = 1 - \left(\frac{e^\epsilon-1}{e^\epsilon+1}\right)^2(1-2\Phi(\theta_0-\theta))^2.
\end{align*}
%Multiplying by $n$ and taking the limit as $n$ tends to infinity, we get
Since the inverse function theorem gives $f'(x) = \frac12 \frac{e^\epsilon+1}{e^\epsilon-1}\left(\phi\circ\Phi^{-1}\left(\frac12 - \frac12 \frac{e^\epsilon+1}{e^\epsilon-1}x\right)\right)^{-1}$ and we have $f(\E[Z_1]) = \theta$, the asymptotic variance of $\tilde{\theta}_n$ is given by
$\frac{1}{4}\left(\frac{e^\epsilon + 1}{e^\epsilon - 1}\right)^2 \frac{1}{\phi (\theta_0 - \theta)^2}\left(1 - \left(\frac{e^\epsilon-1}{e^\epsilon+1}\right)^2(1-2\Phi(\theta_0-\theta))^2\right),
$
which concludes the proof.
\end{proof}

\subsection{General variance}

In this subsection, we describe how our efficient method for mean estimation can be adjusted to the case of a general, but known, variance. Consider original sensitive data $X_1,\dots, X_n$ that are independently and identically distributed as $P_{\theta,\sigma^2}=N(\theta, \sigma^2)$ with a known variance $\sigma^2>0$. Hence, in this subsection, the statistical model is given by $\P_{\sigma^2} = \{N(\theta,\sigma^2):\theta\in\R\}$. The only conceptual challenge is to determine a lower bound on the asymptotic variance of any locally private method. Let 
\begin{equation*}
I_{\theta}(Q\P_{\sigma^2}):=\E_{Z \sim QP_{\theta,\sigma^2}}\left[\left(\frac{\partial \log QP_{\theta,\sigma^2}(Z)}{\partial \theta}\right)^2\right]
\end{equation*}
denote the information about $\theta$ contained in the $Q$-sanitized data when $\sigma^2$ is fixed. In particular, with our previous notation we have $\P_1 = \P$. The following result extends Theorem~\ref{thm:main} to the case of general $\sigma^2>0$.

\begin{theorem}\label{thm:UnknownSigma}
    For $\theta\in\R$, $\sigma^2>0$, $\epsilon>0$ and $Q\in\mathcal Q_\epsilon$ we have
    $$
    I_{\theta}(Q\P_{\sigma^2}) = \frac{1}{\sigma^2}I_{\frac{\theta}{\sigma}}(Q_\sigma\P),
    $$
    for some mechanism $Q_\sigma\in\mathcal Q_\epsilon$ that depends only on $\sigma$. In particular, by Theorem~\ref{thm:main}, if $\epsilon\le 1.04$, it holds that
    \begin{equation}\label{eq:OptVarSigma}
    I_{\theta}(Q\P_\sigma) 
    \le 
    \frac{1}{\sigma^2} I_{\frac{\theta}{\sigma}}\left(Q_{\frac{\theta}{\sigma},\epsilon}^{\text{sgn}}\P\right) 
    =
    \frac{1}{\sigma^2}\frac{2}{\pi} \frac{(e^\epsilon - 1)^2}{(e^\epsilon + 1)^2},
    \end{equation}
    for all $Q\in\mathcal Q_\epsilon$.
\end{theorem}

In accordance with standard statistical intuition, Theorem~\ref{thm:UnknownSigma} now shows that for general known $\sigma^2>0$ we should rescale the original data $X_i/\sigma \thicksim N(\frac{\theta}{\sigma},1)$, estimate the ratio $\frac{\theta}{\sigma}$ using the two-stage procedure in \eqref{eq:est1} and \eqref{eq:est2} based on the sign-mechanism and then scale back by $\sigma$ to achieve the optimal variance given by the inverse of the upper bound in \eqref{eq:OptVarSigma}. It is then clear that this final rescaling by the true $\sigma$ amounts exactly to the additional factor $\sigma^2$ in the asymptotic variance (cf. \eqref{eq:OptVarSigma}).

If the variance $\sigma^2$ is actually unknown and has to be estimated from appropriately sanitized data, things become much more involved. Estimating the vector $(\theta, \sigma^2)$ efficiently appears to be an ill-posed problem under local differential privacy in the sense that there is no maximizing element $Q$ in the partial Loewner ordering of semidefinite matrices which maximizes the $2\times2$ Fisher-Information matrix $I_{\theta,\sigma^2}(Q\P^{(2)})$, $\P^{(2)}:= \{N(\theta,\sigma^2):\theta\in\R, \sigma^2>0\}$. Intuitively, this can be understood by noticing that in order to estimate both $\theta$ and $\sigma^2$ we would have to extract two different features of each sensitive $X_i$ in a locally private way and we have to make a decision on how to distribute the privacy budget $\epsilon$ over these two tasks, while for estimation of a single parameter, we can use the full privacy budget on that one. Even if $\sigma^2$ is only considered to be a nuisance parameter and the goal is still an efficient estimation of $\theta$, it is currently not clear what the appropriate lower bound for the asymptotic variance and an efficient procedure would look like. We defer these important and fundamental issues to future research.

We close this discussion by mentioning that for general unknown $\sigma^2>0$ our first stage estimator \eqref{eq:est1} is not even consistent for $\theta$. However, consistent locally private, yet inefficient, estimators for both $\theta$ and $\sigma^2$ can be constructed using the method of moments and clipping as explained, for instance, in Theorem~4.3 of \citet{Steinberger24}.

\section{Experiments}
\label{sec:sim}

In this section, we provide numerical experiments to investigate the finite sample performance of our two-stage locally private estimation procedure described in \eqref{eq:est1} and \eqref{eq:est2} in dependence on the tuning parameters $\theta_0\in\R$ and $n_1\in\N$, for different sample sizes $n$ and $\epsilon=1$. For the plots, we generated 50000 Montecarlo samples and computed the 95\% confidence intervals via bootstrapping. In the left panel of Figure~\ref{fig:sim}, we plot the scaled mean squared error (MSE) of $\hat{\theta}_n$ as a function of the number $n_1$ of samples in the first stage of our private estimation procedure. From the plots, we see that a proper choice of $n_1$ is crucial to get a small MSE. Especially too small values of $n_1$ must be avoided.

In the right panel of Figure~\ref{fig:sim}, we plot the scaled MSE as a function of the difference $\theta_0-\theta$, where $n_1$ was chosen such that the corresponding MSE function in the left panel is minimized. The results are not surprising. The MSE increases as the initial guess $\theta_0$ moves away from the true value $\theta$. Notice that our results do not provide uniform convergence in $\theta$, and in fact it is known that with differential privacy it is impossible to achieve uniformity over an unbounded parameter space. 
\begin{figure}[t]
    \centering    \includegraphics[scale=0.5]{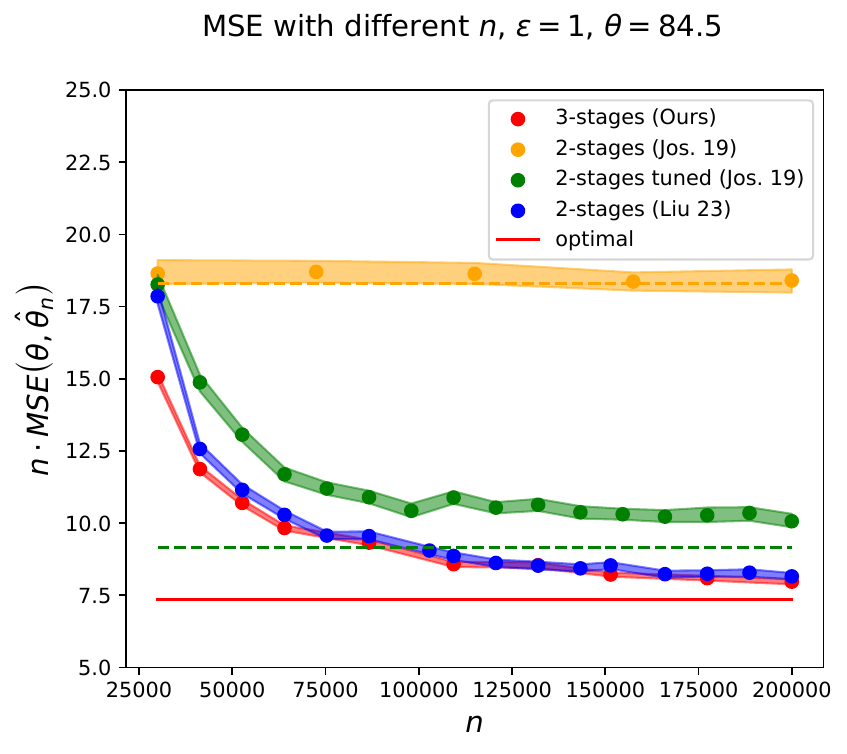}
    \caption{
    The plot compares the performance of our enhanced 3-stage mechanism with the original mechanism from \citet{Jos19}, including a tuned version in which we adjusted the number of samples allocated to the first stage. Additionally, it is compared with \citet{liu2023online}, which has been enhanced by a preliminary stage from \citet{Jos19}. The horizontal lines represent the theoretical variance limits for each mechanism.}
    \label{fig:3_steps}
\end{figure}
\subsection{3-Stage Mechanism}
Estimating the true mean $\theta$ for large initial differences $|\theta - \theta_0|$ is challenging for our two-stage mechanism. To enhance performance, we incorporate the first stage of the mechanism from \citet{Jos19}, which estimates $\theta$ bitwise with high probability, provided an upper bound on the number of bits to estimate. In their approach, the same sign mechanism is used for the second stage. However, our theory predicts that splitting the data equally between these stages is suboptimal. Even when tuned with a better choice of steps allocated to the first stage, their mechanism remains suboptimal because the first stage fails to consistently estimate $\theta$.  

We provide additional results of a 3-stage mechanism in Figure \ref{fig:3_steps}, with $\theta = 84.5$, under the assumption that the unknown $\theta$ is in the range $[0, 128]$. With varying sample sizes $n$, we observe that the MSE converges to the optimal value for our 3-stage procedure. Specifically, we allocate $n_0 = 15,000$ samples to the preliminary stage based on \citet{Jos19}, which we find to be close to optimal for estimating the first 7 bits of $\theta$. We also provide the theoretical limits for the 2-stage mechanisms based on the formula \eqref{eq:one_round_variance} We optimized over the choice of $n_1$ for the 3-stage mechanism and found that for values of $n$ ranging from $30,000$ to $200,000$, the optimal $n_1$ initially grows from $400$ to $700$ before almost saturating around $700$. We conjecture that as $n$ continues to grow, $n_1$ might increase logarithmically; however, we lack the theoretical understanding to confirm this. We also compare our approach with the method of \citet{liu2023online}, enhanced by the preliminary stage of \citet{Jos19}. Our results show that for large $n$, their method performs almost identically to ours, while for smaller $n$, it may perform slightly worse. The method of \citet{liu2023online} has been proven to achieve a variance that we have shown to be asymptotically optimal. We illustrate its performance in Figure \ref{fig:3_steps} as well.

Our results show that the 3-stage mechanism allows us to estimate $\theta$ asymptotically optimally across a broader range of $\theta$ for a finite $n$. 

\section{Discussion}

This study addressed the problem of estimating the unknown mean $\theta$ of a unit variance Gaussian distribution under local differential privacy (LDP). We proved that the sign mechanism, which computes the sign of the difference $X_i - \theta$ and applies randomized response, has the lowest asymptotic variance among all locally differentially private mechanisms in the high-privacy regime ($\epsilon \leq 1.04$). However, this mechanism is purely theoretical since it requires knowledge of $\theta$, which we are trying to estimate. To address this issue, we propose a two-stage mechanism that allows for the estimation of $\theta$ without any assumptions. For practical considerations, we enhanced our mechanism with an additional pre-step from \citet{Jos19} to provide a better initial guess for $\theta$, enabling a good private estimate over a wide range of $\theta$ with a finite sample size $n$.

However, extending our theoretical approach to more complex scenarios presents challenges. For $\epsilon \geq 2$, the binary mechanism becomes provably suboptimal, as evidenced by the Fisher Information, which converges to $\frac{2}{\pi}$ as $\epsilon \to\infty$, rather than to $1$, which is the Fisher Information in classical (non-private) estimation. Identifying exact bounds or closed form expressions of optimal mechanisms for larger $\epsilon$ remains an open question. Additionally, extending our method to the multidimensional case is difficult, as optimizing the inverse Fisher Information matrix is ill-posed. The current theoretical framework also does not apply to minimizing the sum of variances in this context. Future research will need to develop new approaches to address these challenges and fully characterize optimal LDP mechanisms in more complex settings.

\subsubsection*{Acknowledgements}

We would like to express our gratitude to Christoph Lampert for his valuable insights and fruitful discussions that significantly contributed to the development of this paper.

We also thank Salil Vadhan for his constructive feedback on an earlier version of this draft.

The second author gratefully acknowledges support by the Austrian Science Fund (FWF): I~5484-N, as part of the Research Unit 5381 of the German Research Foundation.
%%%%%%%%%%%%%%%%%%%%%%%%%%%%%%%%%%%%%%%%%%%%%%%%%%%%%%%%%%%%

\bibliography{lit}

 %\documentclass[twoside]{article}

%\usepackage{aistats2025}
% If your paper is accepted, change the options for the package
% aistats2025 as follows:
%
%\usepackage[accepted]{aistats2025}
%
% This option will print headings for the title of your paper and
% headings for the authors names, plus a copyright note at the end of
% the first column of the first page.

% If you set papersize explicitly, activate the following three lines:
%\special{papersize = 8.5in, 11in}
%\setlength{\pdfpageheight}{11in}
%\setlength{\pdfpagewidth}{8.5in}

% If you use natbib package, activate the following three lines:
%\usepackage[round]{natbib}
%\renewcommand{\bibname}{References}
%\renewcommand{\bibsection}{\subsubsection*{\bibname}}

% If you use BibTeX in apalike style, activate the following line:
%\bibliographystyle{apalike}

%\begin{document}



% Supplementary material: To improve readability, you must use a single-column format for the supplementary material.
\appendix

\onecolumn
%\aistatstitle{Efficient Estimation of a Gaussian Mean \\ with Local Differential Privacy}

\section{Appendix}

\subsection{Proof of Theorem~\ref{thm:efficiency}}
\label{sec:prof-efficiency}

Rather than relying on Theorem~4.12 in \citet{Steinberger24} and checking all their assumptions, we here present a direct proof that also has the advantage of being self-contained. We begin by showing the consistency of first and second-stage estimators. For convenience, fix $h\in\R$ and set $\theta_n := \theta + h/\sqrt{n}$. All probabilities, expectations, variances and stochastic convergence results below are with respect to $R_{\theta_n}$.

\subsubsection{Consistency}

Showing consistency of the first-stage estimator \eqref{eq:est1} is rather straight forward, because $Z_1,\dots, Z_{n_1}$ are iid with values in $\{-1,1\}$ and 
\begin{align*}
\E[Z_1] &= P(Z_1=1) - P(Z_1=-1)
= 2P(Z_1=1)-1\\
&=2\frac{e^\epsilon}{1 + e^\epsilon} - 2\frac{e^\epsilon - 1}{e^\epsilon + 1}\Phi(\theta_0 - \theta_n) - 1
= \frac{e^\epsilon-1}{e^\epsilon+1}(1-2\Phi(\theta_0-\theta_n)).
\end{align*}

Thus, by Markov's inequality, $\bar{Z}_{n_1}$ converges in probability to $\frac{e^\epsilon-1}{e^\epsilon+1}(1-2\Phi(\theta_0-\theta))\in(-\frac{e^\epsilon-1}{e^\epsilon+1},\frac{e^\epsilon-1}{e^\epsilon+1} )$, and thus $P(|\bar{Z}_{n_1}|<\frac{e^\epsilon-1}{e^\epsilon+1}) \to 1$ as $n\to\infty$ and $\tilde{\theta}_{n_1}$ converges to $\theta$ in probability, by the continuous mapping theorem.

Consistency of \eqref{eq:est2} can be shown in a similar way, using the conditional Markov inequality. In view of consistency of $\tilde{\theta}_{n_1}$ it suffices to show that $\bar{Z}_{n_2}=\frac{1}{n_2}\sum_{i=n_1+1}^nZ_i\to0$ in probability as $n\to\infty$. Conditionally on $\tilde{\theta}_{n_1}$, the $Z_{n_1+1}, \dots, Z_n$ are iid. The conditional expectation is computed in the same way as above, that is 
$$
\E[\bar{Z}_{n_2}|\tilde{\theta}_{n_1}] = \E[Z_n|\tilde{\theta}_{n_1}] = \frac{e^\epsilon-1}{e^\epsilon+1}(1-2\Phi(\tilde{\theta}_{n_1}-\theta_n)),
$$
and it converges to zero in probability as $n\to\infty$. The $Z_i$'s only take values $1$ or $-1$ and thus their conditional variance is bounded by $1$. Now
\begin{align*}
P(|\bar{Z}_{n_2}|>\varepsilon) &\le P(|\bar{Z}_{n_2}-\E[\bar{Z}_{n_2}|\tilde{\theta}_{n_1}]|+|\E[\bar{Z}_{n_2}|\tilde{\theta}_{n_1}]|>\varepsilon) \\
&\le\E\left[P\left(|\bar{Z}_{n_2}-\E[\bar{Z}_{n_2}|\tilde{\theta}_{n_1}]|>\frac{\varepsilon}{2}\Big|\tilde{\theta}_{n_1}\right) \right] +P\left(|\E[\bar{Z}_{n_2}|\tilde{\theta}_{n_1}]|>\frac{\varepsilon}{2}\right) \\
&\le
\frac{4}{\varepsilon^2}\E\left[1\land \Var[\bar{Z}_{n_2}|\tilde{\theta}_{n_1}] \right] + P\left(|\E[Z_n|\tilde{\theta}_{n_1}]|>\frac{\varepsilon}{2}\right)\\
&\le \frac{4}{\varepsilon^2n_2} + P\left(|\E[Z_n|\tilde{\theta}_{n_1}]|>\frac{\varepsilon}{2}\right) \xrightarrow[n\to\infty]{} 0.
\end{align*}

 \subsubsection{Asymptotic normality}
 
For asymptotic normality, first notice that we already showed above that $|\bar{Z}_{n_1}|$ and $|\bar{Z}_{n_2}|$ are bounded by $\frac{e^\epsilon-1}{e^\epsilon+1}$ with asymptotic probability one. Hence, it suffices to proceed on the event where both of these absolute values are appropriately bounded. Write $G(p) := \Phi^{-1}(p)$ and consider the scaled estimation error of \eqref{eq:est2}, that is
\begin{align*}
 \sqrt{n}(\hat{\theta}_n - \theta_n) &= -\sqrt{n}\left( \Phi^{-1}\left(\frac12 - \frac12\frac{e^\epsilon + 1}{e^\epsilon - 1}\bar{Z}_{n_2}\right) - (\tilde{\theta}_{n_1}-\theta_n)\right)\\
 &=-\sqrt{n} \left( G\left(\frac12 - \frac12\frac{e^\epsilon + 1}{e^\epsilon - 1}\bar{Z}_{n_2}\right)  - G(\Phi(\tilde{\theta}_{n_1}-\theta_n))\right).
\end{align*}
By a first order Taylor expansion and writing $A_n:=\frac12 - \frac12\frac{e^\epsilon + 1}{e^\epsilon - 1}\bar{Z}_{n_2} = \frac{1}{n_2}\sum_{i=n_1+1}^{n}\left(\frac12 - \frac12\frac{e^\epsilon + 1}{e^\epsilon - 1}Z_i\right)$ and $B_n:= \Phi(\tilde{\theta}_{n_1}-\theta_n)$, we get
\begin{align}\label{eq:Taylor}
\sqrt{n}(G(A_n) - G(B_n))=  G'(B_n)\sqrt{n}(A_n-B_n) + \frac12 \sqrt{n}\int_{B_n}^{A_n} G''(t)(A_n-t)^2dt.
\end{align}
From our results of the previous subsection we see that $A_n\to\frac12$ and $B_n\to\frac12$ in probability, and since $G'$ is continuous on $[0,1]$ we have $G'(B_n) \to G'(\frac12) = [\phi\circ\Phi^{-1}(\frac12)]^{-1} = \sqrt{2\pi}$, in probability. Next, we show that $\sqrt{n}(A_n-B_n)\to N(0,\frac14(\frac{e^\epsilon+1}{e^\epsilon-1})^2)$ in distribution. For $i=n_1+1,\dots, n$, define $Y_i := \frac12 - \frac12\frac{e^\epsilon + 1}{e^\epsilon - 1}Z_i - \Phi(\tilde{\theta}_{n_1}-\theta_n)$ and note that conditional on $\tilde{\theta}_{n_1}$ they are iid with conditional mean zero and conditional variance
$$
\Var[Y_i|\tilde{\theta}_{n_1}] = \frac14\left(\frac{e^\epsilon + 1}{e^\epsilon - 1} \right)^2 \left[1- \left(\frac{e^\epsilon-1}{e^\epsilon+1}(1-2\Phi(\tilde{\theta}_{n_1}-\theta_n))\right)^2\right] \xrightarrow[n\to\infty]{i.p.} \frac14\left(\frac{e^\epsilon + 1}{e^\epsilon - 1} \right)^2.
$$
Moreover, $\Var[Y_i|\tilde{\theta}_{n_1}] \ge \frac14\left(\frac{e^\epsilon + 1}{e^\epsilon - 1} \right)^2 [1- (\frac{e^\epsilon-1}{e^\epsilon+1})^2]$ and $\E[|Y_i|^3|\tilde{\theta}_{n_1}] \le \frac32 + \frac12\frac{e^\epsilon + 1}{e^\epsilon - 1}$.
Thus, from the Berry-Esseen bound \citep[cf.][Theorem~1]{Berry41} applied to the conditional distribution, we get for every $t\in\R$,
\begin{align*}
&\left|P\left( \sqrt{n_2}\frac{A_n-B_n}{\sqrt{\Var(Y_n|\tilde{\theta}_{n_1})}}\le t\right) - \Phi(t)\right| =
\left|P\left( \sqrt{n_2}\sum_{i=n_1+1}^n\frac{Y_i}{\sqrt{\Var(Y_i|\tilde{\theta}_{n_1})}}\le t\right) - \Phi(t)\right| \\
&\quad\le 
\E\left|P\left( \sqrt{n_2}\sum_{i=n_1+1}^n\frac{Y_i}{\sqrt{\Var(Y_i|\tilde{\theta}_{n_1})}}\le t\Bigg|\tilde{\theta}_{n_1} \right) - \Phi(t)\right| \le\frac{1.88 (\frac32 + \frac12\frac{e^\epsilon + 1}{e^\epsilon - 1})}{\frac18\left(\frac{e^\epsilon + 1}{e^\epsilon - 1} \right)^3 [1- (\frac{e^\epsilon-1}{e^\epsilon+1})^2]^{\frac32}\sqrt{n_2}}.
\end{align*}
Applying Slutski's theorem and noting that $n/n_2\to1$, we get the desired convergence of $\sqrt{n}(A_n-B_n)\to N(0,\frac14(\frac{e^\epsilon+1}{e^\epsilon-1})^2)$ in distribution. Hence, $G'(B_n)\sqrt{n}(A_n-B_n)\to N(0,\frac{\pi}{2}(\frac{e^\epsilon+1}{e^\epsilon-1})^2)$. From Theorem~\ref{thm:main} we see that this is the optimal asymptotic variance. Thus, the proof is finished if we can show that the remainder term in \eqref{eq:Taylor} converges to zero in probability. But this is easily seen, because on the event where $A_n,B_n\in[\frac14,\frac34]$, which has asymptotic probability one, the remainder term is bounded in absolute value by
$$
\frac12|\sqrt{n}(A_n-B_n)| |A_n-B_n|^2 \sup_{t\in[\frac14,\frac34]} |G''(t)| \xrightarrow[n\to\infty]{i.p.}0.
$$

\subsection{Proof of Lemma \ref{lem:a_1_right}}

We need to show that the inequality in Lemma~\ref{lem:a_1_right} holds, i.e., that
\begin{equation}
\label{def:L_lem_1}
L(x, y, t_\epsilon):=  \mathcal{L}\left(1 - \pi y^2, \pi y^2, x, y\right)  = 2 t_\epsilon (2\pi y^2 - 1) + 4 t_\epsilon^2 (y - x) (2 \pi y^2 - 1) - 4 t_\epsilon^2 (y - x)^2 + 1 \ge 0.
\end{equation}
This is a quadratic function in $x$ and we can express it as

\begin{equation*}    
L(x, y, t_\epsilon) = -4t_\epsilon^2 x^2 -4xt_\epsilon^2(2\pi y^2 -1-2y) + 2 t_\epsilon (2\pi y^2 - 1)+ 4t_\epsilon^2 (2\pi y^2 - 1)y - 4t_\epsilon^2 y^2 + 1.
\end{equation*}

%Since $2\pi y^2 - 1 - 2y < 0$ the optimal $x = 0$. Therefore
The minimal value is achieved on the tails, for $x = 0$ or $x = \frac{1}{2}$. First, we consider the case when $x = 0$.

\subsubsection{Case $x = 0$}

\begin{equation}\label{eq:casex0}
L(0, y, t_\epsilon) = 2 t_\epsilon (2\pi y^2 - 1)+ 4t_\epsilon^2 (2\pi y^2 - 1)y - 4t_\epsilon^2 y^2 + 1.
\end{equation}

We will prove that for $t_\epsilon \le \frac{4\pi}{1 + 8\pi} < \frac{1}{2}$ the expression in \eqref{eq:casex0} is greater than $0$.
This expression is a polynomial of the third power in terms of 
$y$, and we minimize it on the segment  $0 \le y \le \frac{1}{2}$. For $y = 0$ we have

\begin{equation}
L(0, 0, t_\epsilon) = 1 - 2t_\epsilon \ge 0,
\end{equation}
which is true for $t_\epsilon \le \frac{1}{2}$. For $y = \frac{1}{2}$, we have:

\begin{equation}
L(0, 1/2, t_\epsilon) = \pi t_\epsilon - 2t_\epsilon + 2 t_\epsilon^2 \left(\frac{\pi}{2} - 1\right) + 1 - t_\epsilon^2 \ge 0,
\end{equation}
which is true for $t_\epsilon \le 1$. Next, we consider the potential optimum in between where the derivative is equal to $0$:

\begin{align}
\frac{\partial L(0, y, t_\epsilon)}{\partial y} &= 8 t_\epsilon \pi y + 24 t_\epsilon^2 \pi y^2 - 4 t_\epsilon^2 - 8 t_\epsilon^2 y\\
&=4t_\epsilon \left(2 \pi y + 6 t_\epsilon \pi y^2 - t_\epsilon - 2 t_\epsilon y\right)\\ 
&=4t_\epsilon \left(6 t_\epsilon \pi y^2 + y (2\pi - 2 t_\epsilon) - t_\epsilon\right) = 0. 
\end{align}
By solving the quadratic equation we obtain an optimal value
\begin{equation}\label{eq:root}
y^{*}(t_\epsilon) = \frac{t_\epsilon - \pi + \sqrt{6\pi t_\epsilon^2 + (\pi - t_\epsilon)^2}}{6\pi t_\epsilon}.
\end{equation}
The other root is negative whereas the root in \eqref{eq:root} is between $0$ and $\frac12$ since
\begin{equation}
    \pi-t_\epsilon\le \sqrt{6\pi t_\epsilon^2 + (\pi - t_\epsilon)^2} \le 3\pi t_\epsilon + (\pi - t_\epsilon).
\end{equation}

Next, we will prove that the function $L(0, y^*(t_\epsilon), t_\epsilon)$ is a decreasing function of $t_\epsilon$. To that end, consider first the derivative of $y^*$, 
\begin{equation}
    \frac{\partial y^*(t_\epsilon)}{\partial t_\epsilon} = \frac{1}{6 t_\epsilon^2}\left[1  - \frac{(\pi t_\epsilon^{-1} - 1)}{ \sqrt{6\pi + (\pi t_\epsilon^{-1} - 1)^2}}\right] > 0,
\end{equation}
and 
\begin{align*}
    \frac{\partial L(0, y^*, t_\epsilon)}{\partial t_\epsilon} &= 2 (2\pi (y^*)^{2} - 1)  + 8 t_\epsilon\pi y^*  \frac{\partial y^*}{\partial t_\epsilon}+ 8t_\epsilon (2\pi (y^*)^2 - 1)y^* + 24t_\epsilon^2 \pi (y^*)^2 \frac{\partial y^*}{\partial t_\epsilon} \\
    &\quad+ 4t_\epsilon^2 (2\pi (y^*)^2 - 1)\frac{\partial y^*}{\partial t_\epsilon} - 8t_\epsilon (y^*)^2 - 8t_\epsilon^2 y^* \frac{\partial y^*}{\partial t_\epsilon}\\
    &= (2\pi (y^*)^2 - 1)(2 + 8t_\epsilon y) - 8 t_\epsilon (y^*)^2 +  \frac{\partial y^*}{\partial t_\epsilon} \left[8\pi t_\epsilon y^* + 24 \pi t_\epsilon^2 (y^*)^2 - 4t_\epsilon^2 - 8t_\epsilon^2 y^*\right]\\
    &= (2\pi (y^*)^2 - 1)(2 + 8t_\epsilon y^*) - 8 t_\epsilon (y^*)^2 +   4t_\epsilon\left[2\pi y^* + 6 \pi t_\epsilon (y^*)^2 - t_\epsilon - 2t_\epsilon y^*\right]\frac{\partial y^*}{\partial t_\epsilon}
\end{align*}
Now we can substitute $y^*(t_\epsilon)$, namely $2\pi (y^*)^2 = \frac{1}{3} - \frac{2y^*}{3t_\epsilon}(\pi - t_\epsilon)$ and $2\pi y^* + 6 \pi t_\epsilon (y^*)^2 - t_\epsilon - 2t_\epsilon y^* = 0$ to obtain

\begin{equation}
    \frac{\partial L(0, y^*(t_\epsilon), t_\epsilon)}{\partial t_\epsilon} = \left(-\frac{2}{3} - \frac{2y^*}{3t_\epsilon}(\pi - t_\epsilon)\right)(2 + 8t_\epsilon y^*) - 8 t_\epsilon (y^*)^2 < 0
\end{equation}
Thus the function $t_\epsilon\mapsto L(0,y^*(t_\epsilon),t_\epsilon)$ is continuous and strictly decreasing. Therefore it suffices to identify the point where it reaches zero. We will argue that it is when $t_\epsilon = \frac{4\pi}{1 + 8\pi}$.  For $t_\epsilon = \frac{4\pi}{1 + 8\pi}$, we have 
\begin{align}
 y^* = \frac{\frac{4\pi}{1 + 8\pi} - \pi + \sqrt{6\pi \left(\frac{4\pi}{1 + 8\pi}\right)^2 + \left(\pi - \frac{4\pi}{1 + 8\pi}\right)^2}}{6\pi \frac{4\pi}{1 + 8\pi}} = \frac{3\pi - 8\pi^2 + \sqrt{96\pi^3 + (8\pi^2 - 3\pi)^2}}{24\pi^2} = \frac{1}{4\pi}.   
\end{align}
By substituting it into the $L(0, y^*(t_\epsilon), t_\epsilon)$ we get
\begin{align}
    L(0, y^*(t_\epsilon), t_\epsilon)\Big|_{t_\epsilon = \frac{4\pi}{1 + 8\pi}} &= \left[2t_\epsilon (2\pi (y^*)^2 - 1)+ 4t_\epsilon^2 (2\pi (y^*)^2 - 1)y^* - 4t_\epsilon^2 (y^*)^2 + 1\right]_{t_\epsilon = \frac{4\pi}{1 + 8\pi}}\\
    &= \frac{8\pi}{1 + 8\pi} \left(\frac{1}{8\pi} - 1\right) + \frac{16\pi}{(1 + 8\pi)^2}\left(\frac{1}{8\pi} - 1\right) - \frac{4}{(1 + 8\pi)^2} + 1\\
    &= \frac{1 - 64\pi^2}{(1 + 8\pi)^2} + \frac{2 - 16\pi}{(1 + 8\pi)^2} - \frac{4}{(1 + 8\pi)^2} + 1\\
    &= \frac{1 - 64\pi^2 + 2 - 16\pi - 4 + 1 + 16\pi + 64\pi^2}{(1 + 8\pi)^2} = 0.
\end{align}
Therefore, the original inequality holds when $t_\epsilon \le \frac{4\pi}{1 + 8\pi}$, which concludes the proof of the case $x = 0$.

\subsubsection{Case $x = \frac{1}{2}$, $0 \le y \le \sqrt{\frac{1}{\pi} - \frac{1}{4}}$}

We need to show the following inequality for the function $L(x, y, t_\epsilon)$ defined in \eqref{def:L_lem_1}:
\begin{align}
    L(1/2, y, t_\epsilon) =2 t_\epsilon (2\pi y^2 - 1) + 4 t_\epsilon^2 (y - 1/2) (2 \pi y^2 - 1)  - 4 t_\epsilon^2 (y - 1/2)^2 + 1  \ge 0 .
\end{align}
First, we note that for $y = 0$,
\begin{equation}
L(1/2, 0, t_\epsilon) = -2t_\epsilon + 2t_\epsilon^2 +1 -t_\epsilon^2 = (t_\epsilon - 1)^2\ge 0.
\end{equation}
Next, for $y =  \sqrt{\frac{1}{\pi} - \frac{1}{4}}$, we have
\begin{align}
    L\left(\frac{1}{2}, \sqrt{\frac{1}{\pi} - \frac{1}{4}}, t_\epsilon\right) &= -t_\epsilon (\pi - 2) + 2 t_\epsilon^2 \left(\frac{1}{2} - \sqrt{\frac{1}{\pi} - \frac{1}{4}}\right) (\pi - 2) + 1 - 4 t_\epsilon^2 \left(\sqrt{\frac{1}{\pi} - \frac{1}{4}} - \frac{1}{2}\right)^2. 
\end{align}
To simplify the expression, let us denote $\alpha = \sqrt{\frac{1}{\pi} - \frac{1}{4}}$. Then the desired inequality is reduced to
\begin{align}\label{eq:ineqLalpha}
L\left(\frac{1}{2}, \alpha, t_\epsilon\right) =  2 t_\epsilon^2 \left(\frac{1}{2} - \alpha\right) \left(\pi - 3 + 2\alpha\right) -t_\epsilon (\pi - 2) + 1 \ge 0.
\end{align}
The minimum of this quadratic function is achieved when $t_\epsilon = \frac{\pi - 2}{2(1 - 2\alpha)(\pi - 3 + 2\alpha)} \approx 1.8$, therefore it is a decreasing function of $t_\epsilon$ on the range $t_\epsilon \le 1$ and $L(1/2, \alpha, 1) \approx 0.175 > 0$, therefore the inequality \eqref{eq:ineqLalpha} is fulfilled for $t_\epsilon \le 1$.
Now we need to compute the derivative of the function with respect to $y$ to find a possible minimal value on the segment $0 \le y \le \sqrt{\frac{1}{\pi} - \frac{1}{4}}$.

\begin{align}
   \frac{\partial L(1/2, y, t_\epsilon)}{\partial y} &= 8 t_\epsilon \pi y + 4 t_\epsilon^2 (2 \pi y^2 - 1) + 16 \pi t_\epsilon^2 (y - 1/2) y  - 8 t_\epsilon^2 (y - 1/2)\\
   &= 4t_\epsilon\left(2\pi y + t_\epsilon 2 \pi y^2 - t_\epsilon + 4\pi t_\epsilon y^2 - 2\pi t_\epsilon y - 2t_\epsilon y + t_\epsilon\right) \\
   &= 8t_\epsilon \left( 3\pi t_\epsilon y^2 + (\pi - \pi t_\epsilon - t_\epsilon)y\right) = 0
\end{align}
The case when $y = 0$ we have already considered, the second root is 
\begin{equation}
y = \frac{t_\epsilon (\pi + 1) - \pi}{3\pi t_\epsilon}.
\end{equation} For $ 0 < t_\epsilon \le \frac{1}{2}$ this fraction is always negative. Therefore, the minimum is achieved when $y = 0$ or $y = \sqrt{\frac{1}{\pi} - \frac{1}{4}}$, concluding the proof of the lemma.

\subsection{Proof of Lemma \ref{lem:a_1_left}}

We need to show that
\begin{align*}
L(x, y, t_\epsilon):=  \mathcal{L}\left(\pi x^2, \pi y^2, x, y\right) = 2 t_\epsilon \pi (y^2 - x^2) &+ 4 t_\epsilon^2 \pi (y - x) (y^2 - x^2) - \pi^2 (x^2 + y^2)^2\\
&\quad+ 2\pi (x^2 + y^2) - 4 t_\epsilon^2 (y - x)^2 \ge 0.
\end{align*}
We consider the function $L$ of two variables $(x, y)$ and parameter $t_\epsilon$. We minimize this function to prove that the minimal value is achieved when $x = y = 0$ and therefore the inequality holds. We start by proving that for $t_\epsilon \le \frac{1}{2}$ the optimal value has to be on the boundary, that is, we show that the optimum can not be achieved in the interior by showing that the partial derivatives cannot simultaneously be equal to $0$. Consider the system of derivatives
\begin{equation*}
    \begin{cases}
        \frac{\partial L(x, y, t_\epsilon)}{\partial y} =  4 t_\epsilon \pi y + 8 t_\epsilon^2 \pi (y^2 - x^2) + 4 t_\epsilon^2\pi (y - x)^2 - 4 \pi^2 (x^2 + y^2)y + 4\pi y - 8 t_\epsilon^2 (y - x) = 0,\\
        \frac{\partial L(x, y, t_\epsilon)}{\partial x} = -4 t_\epsilon \pi x - 8 t_\epsilon^2 \pi (y^2 - x^2) + 4 t_\epsilon^2\pi (y - x)^2 -  4 \pi^2 (x^2 + y^2)x + 4 \pi x +  8t_\epsilon^2 (y - x) = 0.
    \end{cases}
\end{equation*}
We transform this system into the equivalent system of the sum and difference of those equalities, namely,
\begin{equation}
    \begin{cases}
        4 t_\epsilon \pi (y - x) + 8 t_\epsilon^2 \pi (y - x)^2 - 4 \pi^2 (x^2 + y^2)(y + x) + 4\pi (x + y)  = 0,\\
        4 t_\epsilon \pi (y + x) + 16 t_\epsilon^2 \pi (y^2 - x^2) - 4 \pi^2 (x^2 + y^2)(y - x) + 4\pi (y - x) - 16 t_\epsilon^2 (y -x)  = 0.
    \end{cases}
\end{equation}
For simplicity, we divide everything by $4\pi$ and collect terms to get

\begin{equation}
    \begin{cases}
         t_\epsilon (y - x)(  1 + 2 t_\epsilon(y - x)) +(y + x) (1 -   \pi (x^2 + y^2)) = 0,\\
        t_\epsilon (y + x)(  1 + 4 t_\epsilon(y - x)) +(y - x) (1 -   \pi (x^2 + y^2)) - \frac{4}{\pi}t_\epsilon^2 (y - x) = 0.
    \end{cases}
\end{equation}
From the first equation, we conclude that $y - x < 0$, or otherwise the sum cannot be zero as $(y+x)(1 - \pi(x^2 + y^2)) > 0$ in the interior of the domain. From the first equation, we can express $1 -\pi(x^2 + y^2)$ as
\begin{equation}
    1 - \pi(x^2 + y^2) = \frac{t_\epsilon (x - y) (1 + 2 t_\epsilon (y - x))}{x + y},
\end{equation}
substitute it in the second equation and multiply by $x + y > 0$ to get
\begin{equation}
    t_\epsilon (y + x)^2(  1 + 4 t_\epsilon(y - x)) - t_\epsilon(x -y)^2 (  1 + 2 t_\epsilon(y - x)) + \frac{4}{\pi}t_\epsilon^2 (x^2 - y^2) = 0.
\end{equation}
Now we divide everything by $t_\epsilon>0$ and recombine the terms to get
\begin{equation}\label{eq:nonzero}
   \mathrm{G}(x, y, t_\epsilon) :=  t_\epsilon (x - y) \left(2 (x - y)^2 - 4(x + y)^2 + \frac{4}{\pi}(x + y)\right)  + (x + y)^2 - (x - y)^2= 0.
\end{equation}
We will show that for $t_\epsilon \le \frac{1}{2}$ this identity can not hold. In particular, we show that the left side is larger than $0$. To that end, let us consider the change of variables $a = x - y > 0$, $b = x + y \ge a$ and

\begin{equation}
\mathrm{G}(x(a, b), y(a, b), t_\epsilon) = t_\epsilon a\left(2 a^2 - 4b^2 + \frac{4}{\pi}b\right) + b^2 - a^2
\end{equation}

Consider the derivative with respect to $b$,
\begin{equation}
\frac{\partial \mathrm{G}(x(a, b), y(a, b), t_\epsilon)}{\partial b} = 2b + t_\epsilon a \left(\frac{4}{\pi} -8 b\right) = 2b(1 - 4t_\epsilon a) + \frac{4}{\pi}t_\epsilon a > 0
\end{equation}
for $t_\epsilon \le \frac{1}{2}$ and because $a \le \frac{1}{2}$ by assumption.
Therefore $\mathrm{G}(x(a, b), y(a, b), t_\epsilon)$ is minimized when $b$ is minimized, but $b \ge a$, thus
\begin{equation}
\mathrm{G}(x(a, b), y(a, b), t_\epsilon) \ge t_\epsilon a \left(2a^2 - 4a^2 + \frac{4}{\pi}a\right) =  t_\epsilon a^2 \left( \frac{4}{\pi} - 2a\right) > 0.
\end{equation}
Therefore the identity \eqref{eq:nonzero} can not hold and there is no optimum in the interior of the domain. Thus, we need to check the boundaries, which leads us to five cases:

\begin{enumerate}
    \item $x^2 + y^2 = \frac{1}{\pi}:$ 
    \begin{align*}
    L\left(x, \sqrt{\frac{1}{\pi} - x^2}, t_\epsilon\right)  &= 2 t_\epsilon \pi (y^2 - x^2) (1 + 2 t_\epsilon (y - x)) - 4 t_\epsilon^2 (y - x)^2  + 1 \\
    &\ge
    -\frac{3\pi}{2} |y^2 - x^2| - (y - x)^2 + 1\\
    &\ge -\frac{3\pi}{2} \left(\frac{1}{2} - \frac{1}{\pi}\right) - \left(\frac{1}{2} - \sqrt{\frac{1}{\pi} - \frac{1}{4}}\right)^2 + 1 \ge 0.
    \end{align*}
    \item $x = \frac{1}{2}$, $0 \le y \le \sqrt{\frac{1}{\pi} - \frac{1}{4}}$:
    \begin{align*}
        L(1/2, y, t_\epsilon) &= 4t_\epsilon^2 \left(y - \frac{1}{2}\right)^2\left( \pi \left(\frac{1}{2} + y\right) - 1\right) +2t_\epsilon \pi \left(y^2 - \frac{1}{4}\right) - \pi^2 \left(\frac{1}{4} + y^2\right)^2 + 2\pi \left(\frac{1}{4} + y^2\right) \\ 
        & \ge -\frac{\pi}{4} + \pi \left(\frac{1}{4} + y^2\right)  \ge 0.
    \end{align*}
    \item $y = \frac{1}{2}, 0\le x \le \sqrt{\frac{1}{\pi} - \frac{1}{4}}$:
    \begin{align*}
         L(x, 1/2, t_\epsilon) = 4t_\epsilon^2 \left(x - \frac{1}{2}\right)^2\left( \pi \left(\frac{1}{2} + x\right) - 1\right) &+2t_\epsilon \pi \left(\frac{1}{4} - x^2\right)\\
         &\quad- \pi^2 \left(\frac{1}{4} + x^2\right)^2 + 2\pi \left(\frac{1}{4} + x^2\right) \ge 0.
    \end{align*}
    \item $x = 0$:
    \begin{align*}
        L(0, y, t_\epsilon) &= 2 t_\epsilon \pi y^2 + 4t_\epsilon^2 \pi y^3 - \pi^2 y^4 + 2\pi y^2 - 4 t_\epsilon^2 y^2 \\ 
        & = y^2(-\pi^2 y^2 + 4t_\epsilon^2 \pi y + 2\pi t_\epsilon + 2\pi - 4 t_\epsilon^2) \\
        & \ge y^2\left(-\frac{\pi^2}{4} + 2\pi - 1\right) \ge 0.
    \end{align*}
    \item $y = 0$: 
    \begin{align*}
        L(x, 0, t_\epsilon) &= -2 t_\epsilon \pi x^2 + 4 t_\epsilon^2 \pi x^3 - \pi^2 x^4 + 2\pi x^2 - 4 t_\epsilon^2 x^2  \\
        & = x^2(-\pi^2 x^2 + 4t_\epsilon^2 \pi x - 2t_\epsilon \pi + 2\pi - 4t_\epsilon^2) \\
        & \ge x^2( -\frac{\pi^2}{4} + 2\pi t_\epsilon^2 - 2t_\epsilon \pi + 2\pi  - 4 t_\epsilon^2) \ge x^2\left(-\frac{\pi^2}{4} + \frac{\pi}{2} - \pi + 2\pi - 1\right) \ge 0.
    \end{align*}
\end{enumerate}
This concludes the proof. 

\subsection{Proof of Theorem~3}

Let $(\mathcal Z, \mathcal G)$ denote the measurable space on which the mechanism $Q$ generates its outputs, let $p_{\theta,\sigma^2}$ be the density of the normal distribution with mean $\theta$ and variance $\sigma^2$ and write $s_{\theta|\sigma^2}(x) = \frac{x-\theta}{\sigma^2}$ for the score in the model $\P_{\sigma^2}$. In view of Lemma~3.1 and Lemma~D.2, the Fisher-Information can be represented by
\begin{equation}\label{eq:FI-ChOfVar}
    I_{\theta}(Q\P_{\sigma^2}) = \int_{\mathcal Z} t_{\theta|\sigma^2}(z)^2 q_{\theta|\sigma^2}(z)\nu(dz),
\end{equation}
where $\nu(dz) = Q(dz|x_0)$ for some $x_0\in\R$, $q_{\theta|\sigma^2}(z) := \int_\R q(z|x) p_{\theta,\sigma^2}(x)dx$ is a $\nu$-density of the sanitized data $Z_i\thicksim QP_{\theta,\sigma^2}$, $t_{\theta|\sigma^2}(z) := \int_\R s_{\theta|\sigma^2}(x) \frac{q(z|x) p_{\theta,\sigma^2}(x)}{q_{\theta|\sigma^2}(z)}dx$ is the score at $\theta$ in the corresponding statistical model $Q\P_{\sigma^2}$ and $q : \mathcal Z\times\R \to[e^{-\epsilon}, e^{\epsilon}]$ is a measurable function such that $q(z|x)p_{\theta,\sigma^2}(x)$ is a $\nu\otimes\lambda$-density of $Q(dz|x)P_{\theta,\sigma^2}(dx)$.
By a simple change of variable, we get
$$
q_{\theta|\sigma^2}(z) = \int_\R q(z|\sigma x) p_{\frac{\theta}{\sigma},1}(x) dx
$$
and
$$
t_{\theta|\sigma^2}(z) = \frac{1}{\sigma}\int_\R \left(x-\frac{\theta}{\sigma}\right) \frac{q(z|\sigma x) p_{\frac{\theta}{\sigma},1}(x)}{q_{\theta|\sigma^2}(z)} dx = \frac{1}{\sigma} t_{\frac{\theta}{\sigma},1}(z).
$$
Hence, if we define $Q_\sigma(A|x):= Q(A|\sigma x)$, we see that $Q_\sigma\in\mathcal Q_\epsilon$ and that $q(z|\sigma x) p_{\frac{\theta}{\sigma},1}(x)$ is a $\nu\otimes\lambda$-density of $Q_\sigma(dz|x) P_{\frac{\theta}{\sigma},1}(dx)$. In particular, $QP_{\theta,\sigma^2} = Q_\sigma P_{\frac{\theta}{\sigma},1}$ and $t_{\frac{\theta}{\sigma},1}$ is a score at $\frac{\theta}{\sigma}$ in the model $Q_\sigma \P$. Therefore, we recognize \eqref{eq:FI-ChOfVar} to be equal to $\frac{1}{\sigma^2}I_{\frac{\theta}{\sigma}}(Q_\sigma\P)$.\hfill\qed

%\end{document}

\end{document}